\documentclass[12pt]{amsart} 
\usepackage[dvips]{graphicx}
\usepackage{amsmath}
\usepackage{amssymb}
\usepackage{amscd}
\usepackage{mathrsfs}
\usepackage{color}
\usepackage[all]{xy}
\usepackage{mathtools,stmaryrd}
\usepackage{tikz}
\usetikzlibrary{intersections,calc,arrows}
\DeclareFontEncoding{OT2}{}{}
\DeclareFontSubstitution{OT2}{cmr}{m}{n}
\DeclareFontFamily{OT2}{cmr}{\hyphenchar\font45 }
\DeclareFontShape{OT2}{cmr}{m}{n}{
<5><6><7><8><9>gen*wncyr
<10><10.95><12><14.4><17.28><20.74><24.88>wncyr10}{}
\DeclareFontShape{OT2}{cmr}{b}{n}{
<5><6><7><8><9>gen*wncyb
<10><10.95><12><14.4><17.28><20.74><24.88>wncyb10}{}
\DeclareMathAlphabet{\mathcyr}{OT2}{cmr}{m}{n}
\DeclareMathAlphabet{\mathcyb}{OT2}{cmr}{b}{n}
\SetMathAlphabet{\mathcyr}{bold}{OT2}{cmr}{b}{n}
\DeclareMathOperator*{\foo}{\SH}
\theoremstyle{plain}
\newtheorem{theorem}{Theorem}[section]
\newtheorem{proposition}[theorem]{Proposition}
\newtheorem{lemma}[theorem]{Lemma}
\newtheorem{corollary}[theorem]{Corollary}
\newtheorem{conjecture}[theorem]{Conjecture}
\theoremstyle{definition}
\newtheorem{definition}[theorem]{Definition}
\newtheorem{example}[theorem]{Example}
\newtheorem{remark}[theorem]{Remark}

\newcommand{\bZ}{\mathbb{Z}}
\newcommand{\bQ}{\mathbb{Q}}
\newcommand{\bR}{\mathbb{R}}
\newcommand{\cA}{\mathcal{A}}
\newcommand{\wcA}{\widehat{\mathcal{A}}}
\newcommand{\cF}{\mathcal{F}}

\newcommand{\cS}{\mathcal{S}}
\newcommand{\wcS}{\widehat{\mathcal{S}}}
\newcommand{\cZ}{\mathcal{Z}}
\newcommand{\frH}{\mathfrak{H}}

\newcommand{\sh}{\mathbin{\mathcyr{sh}}}
\newcommand{\SH}{\mathbin{\mathcyr{Sh}}}
\newcommand{\bk}{\boldsymbol{k}}
\newcommand{\bl}{\boldsymbol{l}}
\newcommand{\bh}{\boldsymbol{h}}
\newcommand{\pp}{\boldsymbol{p}}

\newcommand{\h}{\mathrm{h}}

\newcommand{\rt}{\mathrm{rt}}
\newcommand{\dep}{\mathrm{dep}}
\newcommand{\wt}{\mathrm{wt}}
\newcommand{\jump}[1]{\ensuremath{\llbracket #1 \rrbracket}}
\newcommand{\MT}{\mathrm{MT}}

\numberwithin{equation}{section}
\address{Multiple Zeta Research Center, Kyushu University, 744, Motooka, Nishi-ku, Fukuoka, 819-0395, Japan}
\email{m-ono@math.kyushu-u.ac.jp}
\address{Research Alliance Center for Mathematical Sciences, Tohoku University, 6-3, Aoba, Aramaki, Aoba-Ku, Sendai, 980-8578, Japan}
\email{shinichiro.seki.b3@tohoku.ac.jp}
\address{Department of Mathematics, Faculty of Science and Technology, Keio University, 3-14-1 Hiyoshi, Kouhoku-ku, 
Yokohama, 223-8522, Japan}
\email{yamashu@math.keio.ac.jp}
\title{Truncated $t$-adic symmetric multiple zeta values and double shuffle relations}
\author{Masataka Ono \and Shin-ichiro Seki \and Shuji Yamamoto}
\thanks{This research was supported in part by JSPS KAKENHI Grant Numbers 26247004, 16J01758, JP16H06336, 18J00151, 18K03221, 18H05233.}
\subjclass[2010]{11M32, 05C05.}
\keywords{$t$-adic symmetric multiple zeta values, double shuffle relation, Kaneko--Zagier's conjecture, multiple zeta values of Mordell--Tornheim type.}
\begin{document}

\maketitle
\begin{abstract}
We study a refinement of the symmetric multiple zeta value, called the $t$-adic symmetric multiple zeta value, 
by considering its finite truncation. 
More precisely, two kinds of regularizations (harmonic and shuffle) give 
two kinds of the $t$-adic symmetric multiple zeta values, 
thus we introduce two kinds of truncations correspondingly. 
Then we show that our truncations tend to the corresponding $t$-adic symmetric multiple zeta values, 
and satisfy the harmonic and shuffle relations, respectively. 
This gives a new proof of the double shuffle relations for $t$-adic symmetric multiple zeta values, 
first proved by Jarossay. 
In order to prove the shuffle relation, we develop the theory of truncated 
$t$-adic symmetric multiple zeta values associated with $2$-colored rooted trees. 
Finally, we discuss a refinement of Kaneko--Zagier's conjecture and 
the $t$-adic symmetric multiple zeta values of Mordell--Tornheim type. 
\end{abstract}
\section{Introduction}
\subsection{Main results}
For a tuple of non-negative integers $\bk=(k_1, \ldots, k_r)$, 
we set $\wt(\bk)\coloneqq k_1+\cdots+k_r$ and $\dep(\bk)\coloneqq r$, 
and call them the \emph{weight} and the \emph{depth} of $\bk$, respectively.
Such $\bk$ is called an \emph{index} if none of its entries is zero. 
In particular, there is a unique index of depth $0$, which we call the \emph{empty index} 
and denote by $\varnothing$. 
An index $\bk=(k_1, \ldots, k_r)$ is said \emph{admissible} if 
$r>0$ and $k_r\geq 2$, or $\bk=\varnothing$. 
For a non-empty admissible index $\bk=(k_1, \ldots, k_r)$, 
the \emph{multiple zeta value} (MZV) $\zeta(\bk)$ is defined by
\[
\zeta(\bk)\coloneqq\sum_{0<n_1<\cdots<n_r}\frac{1}{n^{k_1}_1\cdots n^{k_r}_r}, 
\]
while $\zeta(\varnothing)$ is set to be $1$. 

There are also various variants of MZV. 
One of such variants, called the \emph{$t$-adic symmetric multiple zeta value} 
and denoted by $\zeta^{}_{\wcS}(\bk)$, is the main object of this paper. 
We also call it the \emph{$\wcS$-MZV} for short. 
For the relationship between the $\wcS$-MZV and other variants, see the next subsection.
 
In order to define the $\wcS$-MZV, we first introduce two values $\zeta_{\wcS}^{\ast}(\bk)$ and $\zeta_{\wcS}^{\sh}(\bk)$.
For tuples $\bk=(k_1,\dots, k_r)$ and $\bl=(l_1,\dots, l_r)$, we set 
\begin{gather*}
\bk+\bl\coloneqq (k_1+l_1,\ldots,k_r+l_r),\quad 
b\binom{\bk}{\bl}\coloneqq\prod_{j=1}^r\binom{k_j+l_j-1}{l_j},\\
\overline{\bk}\coloneqq(k_r,\dots, k_1),\quad \bk_{[i]}\coloneqq(k_1,\dots, k_i), 
\quad \bk^{[i]}\coloneqq(k_{i+1},\dots,k_r)\quad (0\leq i\leq r). 
\end{gather*}
Regarding the depth $0$ case, we understand $\varnothing+\varnothing=\varnothing$, 
$b\binom{\varnothing}{\varnothing}=1$, $\overline{\varnothing}=\varnothing$ 
and $\bk_{[0]}=\bk^{[r]}=\varnothing$. 
Let $\cZ$ be the $\bQ$-subalgebra of $\bR$ generated by all MZVs, 
and set $\overline{\cZ}\coloneqq\cZ/\pi^2\cZ$ (recall that $\pi^2=6\zeta(2)\in\cZ$). 
We denote by $\zeta^*$ and $\zeta^{\sh}$ the harmonic and shuffle regularized MZVs, 
respectively (see \S\ref{subsec:Preliminaries} below for the definitions). 
\begin{definition}\label{def:t-adic SMZV}
For $\bullet\in\{\ast,\sh\}$ and an index $\bk=(k_1,\dots,k_r)$, we define 
$\zeta^{\bullet}_{\wcS}(\bk)$ as an element of $\cZ\jump{t}$, with $t$ being an indeterminate, by 
\begin{equation}\label{eq:t-adic SMZV}
\zeta^\bullet_{\wcS}(\bk)\coloneqq\sum_{i=0}^r(-1)^{\wt(\bk^{[i]})}\zeta^{\bullet}(\bk_{[i]})
\sum_{\bl\in\bZ_{\geq 0}^{r-i}}b\binom{\bk^{[i]}}{\bl}\zeta^{\bullet}(\overline{\bk^{[i]}+\bl})t^{\wt(\bl)}.
\end{equation}
\end{definition}
We will prove that $\zeta^{\ast}_{\wcS}(\bk)-\zeta^{\sh}_{\wcS}(\bk)\in \pi^2\cZ\jump{t}$ 
(Proposition \ref{prop:indep-bullet}).
Thus, the following definition is independent of the choice of the regularization 
$\bullet \in \{*, \sh\}$.
\begin{definition}[$\wcS$-MZV]\label{def:wcS-MZV}
For an index $\bk$, we define the \emph{$t$-adic symmetric multiple zeta value} ($\wcS$-MZV)
$\zeta^{}_{\wcS}(\bk)$ by
\[
\zeta^{}_{\wcS}(\bk)\coloneqq\zeta^\bullet_{\wcS}(\bk) \bmod{\pi^2}\in\overline{\cZ}\jump{t}.
\]
\end{definition}

The notion of $\wcS$-MZV, as well as its motivic version, appeared first in the literature 
in Jarossay's article \cite[Definition 1.3]{J4}. He defined it in terms of the Drinfeld associator, 
and later named it the \emph{$\Lambda$-adjoint multiple zeta value}, 
where $\Lambda$ is an indeterminate corresponding to our $t$. 
On the other hand, the second author learned the idea of defining the $\wcS$-MZV 
by the expression \eqref{eq:t-adic SMZV} from Hirose in summer of 2015. 
The fact that Jarossay's $\Lambda$-adjoint MZV and our $\wcS$-MZV coincide up to sign 
was communicated to us by Tasaka, and then by Jarossay himself. 

Jarossay \cite{J5} proved the \emph{double shuffle relation} (DSR) for the $\Lambda$-adjoint MZVs. 
In terms of our $\wcS$-MZVs, it can be stated as follows: 

\begin{theorem}[DSR for $\wcS$-MZVs]\label{DSR for t-adic SMZV}
For any indices $\bk$ and $\bl$, we have 
\begin{align}
&\zeta_{\wcS}^*(\bk*\bl)=\zeta_{\wcS}^*(\bk)\zeta_{\wcS}^*(\bl),\\
&\zeta_{\wcS}^{\sh}(\bk \sh \bl)=(-1)^{\wt(\bl)}\sum_{\bl'\in \bZ^{\dep{(\bl)}}_{\geq0}}
b\binom{\bl}{\bl'}\zeta_{\wcS}^{\sh}(\bk, \overline{\bl+\bl'}){t}^{\wt(\bl')}.
\end{align}
\end{theorem}
Here $\zeta_{\wcS}^*(\bk*\bl)$ (resp.~$\zeta_{\wcS}^{\sh}(\bk\sh\bl)$) is defined as $\sum_{\bh}\zeta_{\wcS}^*(\bh)$ 
(resp.~$\sum_{\bh'}\zeta_{\wcS}^{\sh}(\bh')$) when $z_{\bk}*z_{\bl}=\sum_{\bh}z_{\bh}$ 
(resp.~$z_{\bk}\sh z_{\bl}=\sum_{\bh'}z_{\bh'}$), respectively.
See \S\ref{subsec:Preliminaries} for the definitions of these symbols.
We use the same convention for other zeta values.
In this paper, we give an alternative proof of Theorem \ref{DSR for t-adic SMZV}. 
In fact, we refine the theorem by considering the following truncations. 
\begin{definition}[Truncated $\wcS$-MZVs]\label{def:truncated-sums}
Let $\bk=(k_1, \ldots, k_r)$ be an index and $M$ a positive integer.
We define the \emph{$\ast$-truncated $\wcS$-MZV} $\zeta_{\wcS,M}^{\ast}(\bk)$ by
\[
\zeta^{\ast}_{\wcS, M}(\bk)\coloneqq\sum_{i=0}^r\sum_{\substack{0<n_1<\cdots<n_i<M \\ -M<n_{i+1}<\cdots<n_r<0}}
\frac{1}{n^{k_1}_1\cdots n^{k_i}_i(n_{i+1}+t)^{k_{i+1}}\cdots (n_r+t)^{k_r}} \in \bQ\jump{t}
\]
and the \emph{$\sh$-truncated $\wcS$-MZV} $\zeta_{\wcS,M}^{\sh}(\bk)$ by
\[
\zeta^{\sh}_{\wcS, M}(\bk)
\coloneqq\sum_{i=0}^r\sum_{\substack{0<n_1<\cdots<n_i \\ n_{i+1}<\cdots<n_r<0 \\ n_i-n_{i+1}<M}}
\frac{1}{n^{k_1}_1\cdots n^{k_i}_i(n_{i+1}+t)^{k_{i+1}}\cdots (n_r+t)^{k_r}} \in \bQ\jump{t}.
\]
\end{definition}
The first main result of this paper states that these truncated $\wcS$-MZVs 
tend to the corresponding $\wcS$-MZVs. 
\begin{theorem}[{= Theorem \ref{lim}}]\label{main1}
For any index $\bk$ and $\bullet \in \{\ast, \sh\}$, we have 
\[
\zeta^{\bullet}_{\wcS}(\bk)=\lim_{M\to\infty}\zeta^\bullet_{\wcS, M}(\bk), 
\]
where the limit is taken coefficientwise as the power series in $t$. 
\end{theorem}

Naively speaking, this result says that 
\[\zeta^{\bullet}_{\wcS}(\bk)
=\sum_{i=0}^r\sum_{\substack{0<n_1<\cdots<n_i \\ n_{i+1}<\cdots<n_r<0}}
\frac{1}{n^{k_1}_1\cdots n^{k_i}_i(n_{i+1}+t)^{k_{i+1}}\cdots (n_r+t)^{k_r}}.\]
In practice, the specific choices of partial sums 
$\zeta^{\ast}_{\wcS, M}(\bk)$ and $\zeta^{\sh}_{\wcS, M}(\bk)$ give the limits 
$\zeta^{\ast}_{\wcS}(\bk)$ and $\zeta^{\sh}_{\wcS}(\bk)$, which are different in general. 
Nevertheless, we call Theorem \ref{main1} the \emph{series expression of $\zeta^\bullet_{\wcS}(\bk)$}. 

The second main result is the DSR for the truncated $\wcS$-MZVs, 
which is a refinement of Theorem \ref{DSR for t-adic SMZV} in view of Theorem \ref{main1}. 

\begin{theorem}[DSR for truncated $\wcS$-MZVs]\label{main2}
For any indices $\bk$ and $\bl$ and any positive integer $M$, we have
\begin{align}
&\zeta^{\ast}_{\wcS,M}(\bk\ast\bl)=\zeta^{\ast}_{\wcS,M}(\bk)\zeta^{\ast}_{\wcS,M}(\bl),
\label{truncated harmonic rel}\\
&\zeta^{\sh}_{\wcS,M}(\bk \sh \bl)=(-1)^{\wt(\bl)}\sum_{\bl'\in \bZ^{\dep(\bl)}_{\geq0}}
b\binom{\bl}{\bl'}\zeta^{\sh}_{\wcS,M}(\bk, \overline{\bl+\bl'}){t}^{\wt(\bl')}.
\label{truncated shuffle rel}
\end{align}
\end{theorem}

\begin{remark}
After submitting the first version of this paper to arXiv, 
Jarossay kindly wrote us a letter which explained how our Theorem~\ref{main1} and Theorem~\ref{main2} 
can be interpreted in and deduced from his theory \cite{J5}. 
He also pointed out that some of other results in this paper are particular cases of results in \cite{J5}, 
which we will mention at each point. 
Nevertheless, we consider that it is still worth publishing this paper since 
our concrete description of the truncated $\wcS$-MZVs given in Definition \ref{def:truncated-sums} 
is different from (though interpretable in) Jarossay's language using non-commutative power series, 
our proof of \eqref{truncated shuffle rel} using the theory of $2$-colored rooted trees is original, 
and the application of that theory to $\wcS$-MZV of Mordell--Tornheim type is interesting in itself. 
\end{remark}

\subsection{Background}
Our interest in the $\wcS$-MZVs is motivated by the work of Kaneko and Zagier \cite{KZ} 
concerning two kinds of variants of MZVs, called \emph{$\cA$-finite multiple zeta values} ($\cA$-MZVs) 
and \emph{symmetric multiple zeta values} ($\cS$-MZVs). First we recall the definitions of them. 

Let $\bk=(k_1, \ldots, k_r)$ be any index. For a positive integer $M$, we define 
the \emph{truncated multiple zeta value} $\zeta^{}_M(\bk)$ by
\[
\zeta^{}_M(\bk)\coloneqq\sum_{0<n_1<\cdots<n_r<M}\frac{1}{n^{k_1}_1\cdots n^{k_r}_r},
\]
where we understand $\zeta^{}_M(\varnothing)\coloneqq 1$. 
Then the $\cA$-MZV $\zeta^{}_\cA(\bk)$ and the $\cS$-MZV $\zeta^{}_\cS(\bk)$ are defined by 
\begin{align*}
\zeta^{}_\cA(\bk)&\coloneqq \bigl(\zeta^{}_p(\bk)\bmod p\bigr)_p\in \cA,\\
\zeta^{}_\cS(\bk)&\coloneqq \zeta_\cS^\bullet(\bk) \bmod \pi^2\cZ \in \overline{\cZ}. 
\end{align*}
Here $\cA$ denotes the $\bQ$-algebra 
\[
\cA\coloneqq \prod_{p}\bZ/p\bZ\Biggm/\bigoplus_{p}\bZ/p\bZ 
\]
where $p$ runs over the set of prime numbers, 
and $\zeta_{\cS}^\bullet(\bk)$ denotes the element of $\cZ$ given by 
\[
\zeta_{\cS}^{\bullet}(\bk)\coloneqq 
\sum_{i=0}^r (-1)^{\wt(\bk^{[i]})}\zeta^{\bullet}(\bk_{[i]})\zeta^{\bullet}(\overline{\bk^{[i]}}) 
\]
for $\bullet\in\{*,\sh\}$. 
The definition of $\zeta^{}_{\cS}(\bk)\in\overline{\cZ}$ is independent of the choice of $\bullet\in\{*,\sh\}$, 
since it is shown that $\zeta_{\cS}^*(\bk)\equiv\zeta_{\cS}^{\sh}(\bk)\pmod{\pi^2}$ by Kaneko--Zagier 
(in fact, our Proposition \ref{prop:indep-bullet} generalizes this congruence). 

Kaneko and Zagier conjectured that $\cA$-MZVs and $\cS$-MZVs satisfy exactly the same algebraic relations 
(see Conjecture \ref{Kanenko-Zagier conjecture} for the precise statement). 
For example, they showed that both $\cA$-MZVs and $\cS$-MZVs satisfy the following relations, 
called the \emph{double shuffle relation} (the shuffle relation \eqref{shuffle for FMZV} was also proved 
by the first author \cite[Corollary~4.1]{O} for $\cF=\cA$, and 
by Jarossay \cite[Th\'{e}or\`{e}m 1.7 i)]{J1} and Hirose \cite[Proposition 15]{Hi} for $\cF=\cS$). 

\begin{theorem}[DSR for $\cA$-MZVs and $\cS$-MZVs]
Let $\cF\in\{\cA,\cS\}$. For any indices $\bk$ and $\bl$, we have
\begin{align}
&\zeta_{\cF}^{}(\bk\ast\bl)=\zeta_{\cF}^{}(\bk)\zeta_{\cF}^{}(\bl),\label{harmonic for FMZV}\\
&\zeta_{\cF}^{}(\bk\sh\bl)=(-1)^{\wt(\bl)}\zeta_{\cF}^{}(\bk, \overline{\bl}).\label{shuffle for FMZV}
\end{align}
\end{theorem}
We note that the harmonic relation \eqref{harmonic for FMZV} for each $\cF\in\{\cA,\cS\}$ 
follows from the same relation of truncated MZVs $\zeta^{}_M(\bk*\bl)=\zeta^{}_M(\bk)\zeta^{}_M(\bl)$. 
This is obvious for $\cF=\cA$, while the series expression of $\zeta^*_\cS$ 
(Corollary~\ref{cor:limits} for $\bullet=*$) is needed for $\cF=\cS$. 

Rosen \cite{Ro1} introduced a natural refinement of the $\cA$-MZV in the $\bQ$-algebra 
\[
\wcA\coloneqq\varprojlim_n\Biggl(\prod_{p}\bZ/p^n\bZ\Biggm/\bigoplus_{p}\bZ/p^n\bZ\Biggr). 
\]
Slightly modifying his definition, the second author \cite{S2} defined 
the \emph{$\wcA$-finite multiple zeta value} ($\wcA$-MZV) 
by $\zeta^{}_{\wcA}(\bk)\coloneqq \bigl((\zeta^{}_p(\bk)\bmod{p^n})_p\bigr)_n\in\wcA$. 
Note that $\wcA$ is complete under the $\pp$-adic topology, 
where $\pp$ denotes the element $\bigl((p\bmod{p^n})_p\bigr)_n$ of $\wcA$, 
and $\zeta^{}_{\wcA}(\bk)$ is a lifting of $\zeta^{}_{\cA}(\bk)$ in the sense that 
$\zeta^{}_{\cA}(\bk)=\zeta^{}_{\wcA}(\bk)\bmod{\pp}$ under the canonical isomorphism $\cA\simeq \wcA/\pp\wcA$. 
The DSR for $\cA$-MZVs are extended to the DSR for $\wcA$-MZVs as follows.
\begin{theorem}[DSR for $\wcA$-MZVs]\label{thm:DSR wcA}
For any indices $\bk$ and $\bl$, we have
\begin{align}
\label{eq:harmonic rel hAFMZV}
&\zeta^{}_{\wcA}(\bk\ast\bl)=\zeta^{}_{\wcA}(\bk)\zeta^{}_{\wcA}(\bl),\\
\label{eq:shuffle rel hAFMZV}
&\zeta^{}_{\wcA}(\bk \sh \bl)=(-1)^{\wt(\bl)}\sum_{\bl'\in \bZ^{\dep(\bl)}_{\geq0}}
b\binom{\bl}{\bl'}\zeta^{}_{\wcA}(\bk, \overline{\bl+\bl'}){\pp}^{\wt(\bl')}.
\end{align}
\end{theorem}
Again, the harmonic relation \eqref{eq:harmonic rel hAFMZV} follows immediately from 
the same relation for the truncated MZVs. 
The shuffle relation \eqref{eq:shuffle rel hAFMZV} was proved independently 
by the second author \cite[Theorem 6.4]{S1} and Jarossay \cite[Lemma 4.17]{J4}.

From the perspective of the Kaneko--Zagier conjecture, it is natural to expect that 
there is also some complete algebra with residue ring $\overline{\cZ}$ 
and a lifting of $\zeta^{}_\cS(\bk)$ in that algebra. 
The consistency between Theorem \ref{DSR for t-adic SMZV} and Theorem \ref{thm:DSR wcA} 
strongly suggests that the expected lifting is $\zeta^{}_{\wcS}(\bk)$ 
in the $t$-adically complete algebra $\overline{\cZ}\jump{t}$. 
See \S4 for further discussion on this extended correspondence. 

\subsection{Contents of this paper}
This paper is organized as follows. 

\S\ref{subsec:Preliminaries} provides some preliminaries including the definitions of 
the products $*$ and $\sh$ and the corresponding regularizations. We also prove 
the congruence $\zeta^*_{\wcS}(\bk)\equiv\zeta^{\sh}_{\wcS}(\bk)\mod \pi^2$ there. 
Then the series expressions of $\zeta^{\bullet}_{\wcS}(\bk)$ (Theorem~\ref{main1}) and 
the harmonic relation for $*$-truncated $\wcS$-MZVs (Theorem~\ref{main2} \eqref{truncated harmonic rel}) 
are proved in \S\ref{subsec:lim} and \S\ref{subsec:harmonic rel}, respectively. 

In \S\ref{sec3}, we build a theory of truncated $\wcS$-MZVs associated with \emph{$2$-colored rooted trees}. 
A $2$-colored rooted tree is a combinatorial structure introduced by the first author \cite{O} 
with applications to $\cA$-MZVs including a proof of the shuffle relation \eqref{shuffle for FMZV}. 
We show that a similar argument is applicable in the context of truncated $\wcS$-MZVs. 
In \S\ref{subsec3.1}, we recall the notion of $2$-colored rooted trees and indices on them, 
and define the associated truncated $\wcS$-MZVs. We see that they include 
the $\sh$-truncated $\wcS$-MZVs given in Definition \ref{def:truncated-sums}. 
In \S\ref{subsec3.2}, we prove some basic properties of the truncated $\wcS$-MZVs 
associated with $2$-colored rooted trees. 
These properties are used in \S\ref{subsec3.3} to show the shuffle relation 
for $\sh$-truncated $\wcS$-MZVs (Theorem~\ref{main2} \eqref{truncated shuffle rel}). 
In \S\ref{subsec3.4}, we consider a fairly general class of truncated $\wcS$-MZVs 
associated with $2$-colored rooted trees, and establish an algorithm for representing those values 
in terms of the $\sh$-truncated $\wcS$-MZVs in the sense of Definition \ref{def:truncated-sums}. 

In \S\ref{subsec:KZ-conj}, we briefly discuss the Kaneko--Zagier conjecture on the correspondence 
between $\zeta^{}_{\cA}(\bk)$ and $\zeta^{}_{\cS}(\bk)$, 
and its refinement to that of $\zeta_{\wcA}(\bk)$ and $\zeta_{\wcS}(\bk)$. 
We also recall a relevant theorem of Yasuda that $\zeta^\bullet_{\cS}(\bk)$ generates $\cZ$, 
and an analogous result for $\zeta^\bullet_{\wcS}(\bk)$ due to Jarossay.
Finally, in \S\ref{subsec:MT}, we introduce the \emph{$\wcS$-MZV of Mordell--Tornheim type} 
which corresponds to \emph{$\wcA$-MZV of Mordell--Tornheim type} under the refined Kaneko--Zagier conjecture.
We prove some formulas on them by applying the theory developed in \S\ref{sec3}.

There are some recent works related with the contents of our paper.
Hirose, Murahara and the first author \cite{HMO} define the star-version of $\wcS$-MZVs ($\wcS$-MZSVs) 
and prove cyclic sum formulas for $\wcS$-MZ(S)Vs.
The first and second authors and Sakurada \cite{OSS} calculate some special values and prove sum formulas 
and the Bowman--Bradley type theorem for $\wcS$-MZ(S)Vs modulo $t^n$ ($n=2,3$).
Komori \cite{Ko} defines and studies the unified multiple zeta functions which interpolate $\wcS$-MZVs.
Bachmann, Takeyama and Tasaka \cite{BTT} defines and studies the symmetric Mordell--Tornheim 
multiple zeta values in a way different from ours.
\subsection*{Acknowledgments}
The authors would like to thank Dr.~Minoru Hirose for communicating his idea on the definition of $\wcS$-MZV. 
They also would like to thank Dr.~Yuta Suzuki for helpful comments and 
useful discussion on Proposition~\ref{prop:GYT}. 
They would like to express their sincere gratitude to Prof.~Koji Tasaka 
for informing us about Jarossay's work on the $\Lambda$-adjoint multiple zeta values, 
and to Dr.~David Jarossay for explaining in detail the relationship of his work with ours. 
\section{Series expressions of $\wcS$-MZVs}\label{sec:2}
\subsection{Preliminaries}\label{subsec:Preliminaries}
Let $\frH\coloneqq\bQ\langle x,y \rangle$ be the non-commutative polynomial ring 
over $\bQ$ with two variables $x$ and $y$.
We define two $\bQ$-subalgebras $\frH^0$ and $\frH^1$ of $\frH$ by 
$\frH^1\coloneqq\bQ+y\frH \supset \frH^0\coloneqq\bQ+y\frH x$.
Putting $z_k\coloneqq yx^{k-1}$ for any positive integer $k$, 
we see that $\frH^1$ has a basis $\{z_{\bk} \mid \text{$\bk$: index}\}$ 
consisting of monomials $z_{\bk}\coloneqq z_{k_1}\cdots z_{k_r}$ for 
all indices $\bk=(k_1,\dots,k_r)$, including $z_\varnothing\coloneqq 1$. 
Note also that $\frH^0$ is spanned by the subset $\{z_{\bk} \mid \bk\colon \text{admissible index}\}$. 

The \emph{harmonic product} $\ast$ on $\frH^1$ and the \emph{shuffle product} $\sh$ on $\frH$ 
are defined as follows.
First, we define the $\bQ$-bilinear map $\ast\colon\frH^1\times\frH^1\rightarrow\frH^1$ by the following rules:
\begin{enumerate}
\item $w\ast1=1\ast w=w$ for all $w \in \frH^1$,
\item $(w_1z_{k_1})\ast(w_2z_{k_2})=(w_1\ast w_2z_{k_2})z_{k_1}+(w_1z_{k_1}\ast w_2)z_{k_2}+(w_1\ast w_2)z_{k_1+k_2}$ for all $w_1, w_2 \in \frH^1$ and positive integers $k_1, k_2$.
\end{enumerate}
We similarly define the $\bQ$-bilinear map $\sh\colon\frH\times\frH\rightarrow\frH$ by the following rules:
\begin{enumerate}
\item $w \sh1=1\sh w=w$ for all $w \in \frH$,
\item $(w_1u_1)\sh(w_2u_2)=(w_1\sh w_2u_2)u_1+(w_1u_1\sh w_2)u_2$ for all $w_1, w_2 \in \frH$ and $u_1, u_2 \in \{x,y\}$.
\end{enumerate}

It is known that $\frH^1$ (resp.~$\frH$) becomes a commutative $\bQ$-algebra with respect to 
the multiplication $\ast$ (resp.~$\sh$), which is denoted by $\frH^1_{\ast}$ (resp.~$\frH_{\sh}$).
Then, the subspace $\frH^0$ of $\frH^1$ is closed under $\ast$ and becomes a $\bQ$-subalgebra of $\frH^1_{\ast}$.
Similarly, the subspaces $\frH^0$ and $\frH^1$ of $\frH$ are closed under $\sh$ and become $\bQ$-subalgebras of $\frH_{\sh}$.
$\frH^0_{\ast}$, $\frH^0_{\sh}$ and $\frH^1_{\sh}$ denote these subalgebras, respectively.

We define a $\bQ$-linear map $Z\colon\frH^0 \rightarrow \bR$ by $Z(z_{\bk})\coloneqq\zeta(\bk)$. 
Similarly, for a positive integer $M$, we define a $\bQ$-linear map $Z_{M}\colon\frH^1 \rightarrow \bQ$ 
by $Z_M(z_{\bk})\coloneqq\zeta_M^{}(\bk)$. 
Then we have 
\begin{alignat*}{2}
    Z(w_1\ast w_2)&=Z(w_1)Z(w_2)=Z(w_1\sh w_2) &\quad &\text{for any $w_1,w_2\in\frH^0$, and}\\
    Z_M(w_1\ast w_2)&=Z_M(w_1)Z_M(w_2) &\quad &\text{for any $w_1,w_2\in\frH^1$.} 
\end{alignat*}

Let $T$ be an indeterminate.
Since $\frH^1_{\ast}\cong\frH^0_{\ast}[y]$ and $\frH^1_{\sh}\cong\frH^0_{\sh}[y]$ 
(see \cite[Theorem 3.1]{Ho} and \cite[Theorem 6.1]{Re}), we can uniquely extend the map $Z$ 
to $\bQ$-algebra homomorphisms $Z^\ast\colon\frH^1_\ast\rightarrow\cZ[T]$ and 
$Z^{\sh}\colon\frH^1_{\sh}\rightarrow\cZ[T]$ satisfying $Z^\ast(y)=Z^{\sh}(y)=T$.
For an index $\bk$ and $\bullet\in\{\ast,\sh\}$, we set $\zeta^\bullet(\bk;T)\coloneqq Z^\bullet(z_{\bk})$ 
and $\zeta^\bullet(\bk)\coloneqq\zeta^\bullet(\bk;0)$. 
We call $\zeta^\ast(\bk)$ (resp. $\zeta^{\sh}(\bk)$) the \emph{harmonic} (resp.~\emph{shuffle}) 
\emph{regularized MZV}.
See \cite{IKZ} for details of the theory of regularization of multiple zeta values.

Now the expression 
\[
\zeta^\bullet_{\wcS}(\bk)=\sum_{i=0}^r(-1)^{\wt(\bk^{[i]})}\zeta^{\bullet}(\bk_{[i]})
\sum_{\bl\in\bZ_{\geq 0}^{r-i}}b\binom{\bk^{[i]}}{\bl}\zeta^{\bullet}(\overline{\bk^{[i]}+\bl})t^{\wt(\bl)}
\]
in Definition \ref{def:t-adic SMZV} makes sense.
To justify Definition~\ref{def:wcS-MZV}, we need the following congruence. 

\begin{proposition}\label{prop:indep-bullet}
For any index $\bk$, we have 
\[
\zeta_{\wcS}^*(\bk)\equiv\zeta_{\wcS}^{\sh}(\bk)\mod{\pi^2\cZ\jump{t}}. 
\]
\end{proposition}
\begin{proof}
First we note that, for $\bullet\in\{*,\sh\}$, 
the definition of $\zeta_{\wcS}^\bullet$ can be rewritten as 
\[
\zeta_{\wcS}^\bullet(\bk)
=\sum_{\bl=(l_1,\ldots,l_r)\in\bZ_{\geq 0}^r} b\binom{\bk}{\bl}t^{\wt(\bl)}
\sum_{\substack{0\leq i\leq r\\\text{ s.t. } \bl_{[i]}=(0,\ldots,0)}}(-1)^{\wt(\bk^{[i]})}
\zeta^\bullet(\bk_{[i]})\zeta^\bullet(\overline{\bk^{[i]}+\bl^{[i]}}). 
\]
From this expression, we see that the coefficient of each $t^n$ in the difference 
$\zeta_{\wcS}^*(\bk)-\zeta_{\wcS}^{\sh}(\bk)$ is a $\bZ$-linear combination 
of terms of the form 
\begin{equation}\label{eq:indep-bullet claim}
\sum_{j=0}^s(-1)^{s-j}\Bigl\{
\zeta^*(\bh,\underbrace{1,\ldots,1}_{j})\zeta^*(\bh',\underbrace{1,\ldots,1}_{s-j})
-\zeta^{\sh}(\bh,\underbrace{1,\ldots,1}_{j})\zeta^{\sh}(\bh',\underbrace{1,\ldots,1}_{s-j})
\Bigr\}
\end{equation}
for some admissible indices $\bh$, $\bh'$ and integer $s\geq 0$. 
To prove that such terms are contained in $\pi^2\cZ$, we consider the generating functions 
\[
f^\bullet(\bh,T,x)\coloneqq \sum_{s=0}^\infty \zeta^\bullet(\bh,\underbrace{1,\ldots,1}_s;T)x^s
\]
for any admissible index $\bh$ and $\bullet\in\{*,\sh\}$. 
Then it suffices to show that 
\[
f^*(\bh,0,x)f^*(\bh',0,-x)-f^{\sh}(\bh,0,x)f^{\sh}(\bh',0,-x)\in \pi^2\cZ\jump{x}, 
\]
since the expression \eqref{eq:indep-bullet claim} is the coefficient of $x^s$ 
in this power series. 

By \cite[Theorem 1]{IKZ} and \cite[Proposition 10]{IKZ}, we see that 
\[
f^*(\bh,T,x)=\rho^{-1}\bigl(f^{\sh}(\bh,T,x)\bigr)
=\rho^{-1}\bigl(e^{Tx}f^{\sh}(\bh,0,x)\bigr)
=e^{-\gamma x}\Gamma(1+x)^{-1}e^{Tx}f^{\sh}(\bh,0,x),
\]
where $\rho$ is an $\bR$-linear endomorphism of $\bR[T]$ determined by 
$\rho(e^{Tx})=e^{\gamma x}\Gamma(1+x)e^{Tx}$ (see \cite[(2.2)]{IKZ}). 
Thus we obtain 
\[
f^*(\bh,0,x)f^*(\bh',0,-x)
=\Gamma(1+x)^{-1}\Gamma(1-x)^{-1}f^{\sh}(\bh,0,x)f^{\sh}(\bh',0,-x), 
\]
and the proof is complete since we have 
\[
\Gamma(1+x)^{-1}\Gamma(1-x)^{-1}=\frac{\sin \pi x}{\pi x}\equiv 1 \mod \pi^2\cZ\jump{x}, 
\]
as is well-known. 
\end{proof}

\begin{remark}
This Proposition is indeed a particular case of \cite[Proposition~3.2.4]{J5}. 
Our proof is also essentially the same as Jarossay's, but we write it down here for the convenience of the reader. 
\end{remark}

\subsection{$\ast$- and $\sh$-truncated $\wcS$-MZVs and their limits}\label{subsec:lim}
In this subsection, we prove Theorem \ref{main1} (= Theorem \ref{lim}), which is our first main result. 
Recall the definitions of $\ast$- and $\sh$-truncated $\wcS$-MZVs (Definition~\ref{def:truncated-sums}):
\begin{align*}
&\zeta^{\ast}_{\wcS, M}(\bk)
=\sum_{i=0}^r\sum_{\substack{0<n_1<\cdots<n_i<M \\ -M<n_{i+1}<\cdots<n_r<0}}
\frac{1}{n^{k_1}_1\cdots n^{k_i}_i(n_{i+1}+t)^{k_{i+1}}\cdots (n_r+t)^{k_r}} \in \bQ\jump{t},\\
&\zeta^{\sh}_{\wcS, M}(\bk)
=\sum_{i=0}^r\sum_{\substack{0<n_1<\cdots<n_i \\ n_{i+1}<\cdots<n_r<0 \\ n_i-n_{i+1}<M}}
\frac{1}{n^{k_1}_1\cdots n^{k_i}_i(n_{i+1}+t)^{k_{i+1}}\cdots (n_r+t)^{k_r}} \in \bQ\jump{t}.
\end{align*}

Here and in what follows, the letter $M$ denotes an arbitrary positive integer unless otherwise noted. 

\begin{definition}\label{w}
For an index $\bk=(k_1, \ldots, k_r)$, a non-negative integer $n$ and $\bullet \in \{\ast, \sh\}$, 
we define an element $w^\bullet_{\wcS, n}(\bk)$ of $\frH^1$ by
\[
w^{\bullet}_{\wcS, n}(\bk)\coloneqq\sum_{i=0}^r(-1)^{\wt(\bk^{[i]})}z_{\bk_{[i]}}\bullet 
\sum_{\bl\in\bZ_{\geq 0}^{r-i},\wt(\bl)=n}b\binom{\bk^{[i]}}{\bl}z_{\overline{\bk^{[i]}+\bl}}.
\]
\end{definition}
\begin{proposition}\label{w in H0}
The above $w^{\bullet}_{\wcS, n}(\bk)$ is always an element of $\frH^0$.
\end{proposition}
\begin{proof}
First we consider the case of $\bullet=*$. 
For $0<i<r$ and $\bl=(l_{i+1},\ldots,l_r)\in\bZ_{\geq 0}^{r-i}$, we have 
\[z_{\bk_{[i]}}*z_{\overline{\bk^{[i]}+\bl}}
=\bigl(z_{\bk_{[i-1]}}\ast z_{\overline{\bk^{[i]}+\bl}}\bigr)z_{k_i}
+\bigl(z_{\bk_{[i]}}\ast z_{\overline{\bk^{[i+1]}+\bl^{[1]}}}\bigr)z_{k_{i+1}+l_{i+1}}, \]
the second term of which belongs to $\frH^0$ whenever $l_{i+1}>0$. 
This implies a congruence 
\[z_{\bk_{[i]}}*\sum_{\bl\in\bZ_{\geq 0}^{r-i},\wt(\bl)=n}b\binom{\bk^{[i]}}{\bl}
z_{\overline{\bk^{[i]}+\bl}}\equiv E_i+E_{i+1}\mod \frH^0, \]
where 
\[E_i\coloneqq
\sum_{\bl\in\bZ_{\geq 0}^{r-i},\wt(\bl)=n} b\binom{\bk^{[i]}}{\bl}
\bigl(z_{\bk_{[i-1]}}\ast z_{\overline{\bk^{[i]}+\bl}}\bigr)z_{k_i}. \]
This also holds for $i=0$ and $i=r$ if we put $E_0=E_{r+1}\coloneqq 0$. Hence we have 
\[w^{*}_{\wcS, n}(\bk)\equiv \sum_{i=0}^r(-1)^{\wt(\bk^{[i]})}(E_i+E_{i+1})
=\sum_{i=1}^r(-1)^{\wt(\bk^{[i]})}(1+(-1)^{k_i})E_i. \]
Note that $E_i\in \frH^0$ if $k_i\geq 2$, while $1+(-1)^{k_i}=0$ if $k_i=1$. 
This completes the proof for $\bullet=*$. 

Next we treat the case of $\bullet=\sh$. First note that the sum 
\[\sum_{\bl\in\bZ_{\geq 0}^{r-i},\wt(\bl)=n}b\binom{\bk^{[i]}}{\bl}z_{\overline{\bk^{[i]}+\bl}} \]
is obtained by expanding $y(x^{k_r-1}z_{k_{r-1}}\cdots z_{k_{i+1}}\sh x^n)$ 
to a sum of monomials. Thus we can write $w^{\sh}_{\wcS, n}(\bk)$ as 
\begin{equation}
w^{\sh}_{\wcS, n}(\bk)=\sum_{i=0}^{r}(-1)^{\wt(\bk^{[i]})}z_{\bk_{[i]}}
\sh y(x^{k_r-1}z_{k_{r-1}}\cdots z_{k_{i+1}}\sh x^n)
=\sum_{i=0}^{r}(-1)^{\wt(\bk^{[i]})}(u_i+v_i). 
\end{equation}
Here we put 
\[u_0\coloneqq y(x^{k_r-1}z_{k_{r-1}}\cdots z_{k_1}\sh x^n), \quad 
v_0\coloneqq 0, \quad u_r\coloneqq 0,\quad 
v_r\coloneqq\begin{cases}
z_{\bk} & (n=0), \\
0 & (n>0), 
\end{cases}\]
and define $u_i$ (resp.~$v_i$) ($0<i<r$) to be the partial sum of 
the expansion of 
$z_{\bk_{[i]}}\sh y(x^{k_r-1}z_{k_{r-1}}\cdots z_{k_{i+1}}\sh x^n)$ 
consisting of monomials in which 
the rightmost $y$ in $z_{\bk_{[i]}}$ lies to the left (resp.~the right) of 
the rightmost $y$ in $y(x^{k_r-1}z_{k_{r-1}}\cdots z_{k_{i+1}}\sh x^n)$. 
Then we can check that
\begin{itemize}
\item $u_{i-1}, v_i \in \frH^0$ if  $k_i\geq2$,
\item $u_{i-1}\equiv v_i \pmod{\frH^0}$ if $k_i=1$
\end{itemize}
for $i=1,\ldots,r$. 
Thus we obtain 
\[
w^{\sh}_{\wcS, n}(\bk)=\sum_{i=1}^{r}(-1)^{\wt(\bk^{[i]})}
\bigl((-1)^{k_i}u_{i-1}+v_i\bigr)\equiv 0\pmod{\frH^0}.\qedhere
\]
\end{proof}
We remark that Komori \cite{Ko} also proved Proposition~\ref{w in H0} for $\bullet=*$, independently. 

\begin{lemma}\label{lem:partial fraction shuffle}
For any indices $\bk=(k_1,\ldots,k_r)$ and $\bl=(l_1,\ldots,l_s)$, we have 
\[
\sum_{\substack{0=m_0<m_1<\cdots<m_r\\ 0=n_0<n_1<\cdots<n_s\\ m_r+n_s<M}}
\frac{1}{m_1^{k_1}\cdots m_r^{k_r}n_1^{l_1}\cdots n_s^{l_s}}=Z_M(z_{\bk}\sh z_{\bl}). 
\]
\end{lemma}
\begin{proof}
This result is a kind of folklore, and can be proven in several ways.
For example, this is a consequence of the shuffle relation of multiple polylogarithms,
as shown in \cite[Theorem~8.1]{Kan} and \cite[Proposition~3.4.3]{J5} for prime or prime power $M$. 
Alternatively, we can show the Lemma by applying the partial fraction decomposition 
\[
\frac{1}{ab}=\frac{1}{a(a+b)}+\frac{1}{b(a+b)}
\]
repeatedly; see \cite{KMT} or \cite{O} for a similar argument. 

Here we give a proof which uses an identity in \cite[Lemma~2.12]{Gon} or \cite[Lemma 2.1]{Yam}. 
Set $k\coloneqq k_1+\cdots+k_r$ and $l\coloneqq l_1+\cdots+l_s$, and 
\[
\Sigma_{k,l}\coloneqq \bigl\{\sigma\in\mathfrak{S}_{k+l}\bigm| \sigma(1)<\cdots<\sigma(k),\,
\sigma(k+1)<\cdots<\sigma(k+l)\bigr\}, 
\]
where $\mathfrak{S}_{k+l}$ denotes the group of permutations on the set $\{1,\ldots,k+l\}$. 
If we write $z_{\bk}=u_1\cdots u_k$ and $z_{\bl}=u_{k+1}\cdots u_{k+l}$ with $u_i\in\{x,y\}$, 
then we have 
\[
z_{\bk}\sh z_{\bl}=\sum_{\sigma\in\Sigma_{k,l}}u_{\sigma^{-1}(1)}\cdots u_{\sigma^{-1}(k+l)}.
\]
Moreover, we consider a set $P_{\bk,\bl}^M$ of tuples $(p_1,\ldots,p_{k+l})$ of non-negative integers 
such that $p_1+\cdots+p_{k+l}<M$ and $p_i>0$ if and only if $u_i=y$. Then we see that 
\begin{multline*}
\sum_{\substack{0=m_0<m_1<\cdots<m_r\\ 0=n_0<n_1<\cdots<n_s\\ m_r+n_s<M}}
\frac{1}{m_1^{k_1}\cdots m_r^{k_r}n_1^{l_1}\cdots n_s^{l_s}}\\
=\sum_{(p_1,\ldots,p_{k+l})\in P_{\bk,\bl}^M}\frac{1}{p_1(p_1+p_2)\cdots(p_1+\cdots+p_k)\cdot 
p_{k+1}(p_{k+1}+p_{k+2})\cdots (p_{k+1}+\cdots+p_{k+l})}
\end{multline*}
and 
\begin{multline*}
Z_M(u_{\sigma^{-1}(1)}\cdots u_{\sigma^{-1}(k+l)})\\
=\sum_{(p_1,\ldots,p_{k+l})\in P_{\bk,\bl}^M}
\frac{1}{p_{\sigma^{-1}(1)}(p_{\sigma^{-1}(1)}+p_{\sigma^{-1}(2)})\cdots
(p_{\sigma^{-1}(1)}+\cdots+p_{\sigma^{-1}(k+l)})}. 
\end{multline*}
Thus our claim follows from the identity 
\begin{multline*}
\frac{1}{p_1(p_1+p_2)\cdots(p_1+\cdots+p_k)\cdot 
p_{k+1}(p_{k+1}+p_{k+2})\cdots (p_{k+1}+\cdots+p_{k+l})}\\
=\sum_{\sigma\in\Sigma_{k,l}}
\frac{1}{p_{\sigma^{-1}(1)}(p_{\sigma^{-1}(1)}+p_{\sigma^{-1}(2)})\cdots
(p_{\sigma^{-1}(1)}+\cdots+p_{\sigma^{-1}(k+l)})}, 
\end{multline*}
which is \cite[Lemma~2.12]{Gon} (where $\sigma$ should be $\sigma^{-1}$), 
and is a special case of \cite[Lemma 2.1]{Yam}. 
\end{proof}

\begin{theorem}[= Theorem \ref{main1}]\label{lim}
For any index $\bk$ and $\bullet\in\{*,\sh\}$, we have 
\[
\zeta^\bullet_{\wcS}(\bk)=\lim_{M\to\infty}\zeta^\bullet_{\wcS,M}(\bk), 
\]
where the limit is taken coefficientwise as the power series in $t$. 
\end{theorem}
\begin{proof}
First we prove 
\begin{equation}\label{Taylor exp}
\zeta^{\bullet}_{\wcS, M}(\bk)=\sum_{n=0}^\infty Z_{M}\bigl(w^{\bullet}_{\wcS, n}(\bk)\bigr)t^n. 
\end{equation}
For $\bullet=*$, by using 
\[
\frac{1}{(t-n)^k}=(-1)^k\sum_{l=0}^{\infty}\binom{k+l-1}{l}\frac{t^l}{n^{k+l}},
\]
we can expand $\zeta^{*}_{\wcS, M}(\bk)$ as 
\begin{align*}
&\zeta^{*}_{\wcS, M}(\bk)\\
&=\sum_{i=0}^r
\sum_{\substack{0<n_1<\cdots<n_i<M \\ 0<n_r<\cdots<n_{i+1}<M}}
\frac{1}{n^{k_1}_1\cdots n^{k_i}_i(t-n_r)^{k_r}\cdots(t-n_{i+1})^{k_{i+1}}}\\
&=\sum_{i=0}^r(-1)^{\wt(\bk^{[i]})}
\sum_{\bl=(l_{i+1}, \ldots, l_r)\in\bZ_{\geq 0}^{r-i}}
b\binom{\bk^{[i]}}{\bl}
\sum_{\substack{0<n_1<\cdots<n_i<M \\ 0<n_r<\cdots<n_{i+1}<M}}
\frac{t^{l_{i+1}+\cdots+l_r}}{n^{k_1}_1\cdots n^{k_i}_in^{k_r+l_r}_r\cdots n^{k_{i+1}+l_{i+1}}_{i+1}}. 
\end{align*}
Then the formula \eqref{Taylor exp} for $\bullet=*$ follows from 
\[
\sum_{\substack{0<n_1<\cdots<n_i<M \\ 0<n_r<\cdots<n_{i+1}<M}}
\frac{1}{n^{k_1}_1\cdots n^{k_i}_in^{k_r+l_r}_r\cdots n^{k_{i+1}+l_{i+1}}_{i+1}}
=Z_M(z_{\bk_{[i]}})Z_M(z_{\overline{\bk^{[i]}+\bl}})
=Z_M\bigl(z_{\bk_{[i]}}\ast z_{\overline{\bk^{[i]}+\bl}}\bigr)
\]
for each $0 \leq i\leq r$ and $\bl=(l_{i+1}, \ldots, l_r) \in \bZ^{r-i}_{\geq0}$.
Next we consider $\bullet=\sh$. In this case, we have 
\[
\zeta^{\sh}_{\wcS, M}(\bk)=\sum_{i=0}^r(-1)^{\wt(\bk^{[i]})}
\sum_{\bl=(l_{i+1}, \ldots, l_r)\in\bZ_{\geq 0}^{r-i}} b\binom{\bk^{[i]}}{\bl}
\sum_{\substack{0<n_1<\cdots<n_i \\ 0<n_r<\cdots<n_{i+1} \\ n_i+n_{i+1}<M}}
\frac{t^{l_{i+1}+\cdots+l_r}}{n^{k_1}_1\cdots n^{k_i}_in^{k_r+l_r}_r\cdots n^{k_{i+1}+l_{i+1}}_{i+1}}.
\]
For $0\leq i \leq r$, we have 
\[
\sum_{\substack{0<n_1<\cdots<n_i \\ 0<n_r<\cdots<n_{i+1} \\ n_i+n_{i+1}<M}}
\frac{1}{n^{k_1}_1\cdots n^{k_i}_in^{k_r+l_r}_r\cdots n^{k_{i+1}+l_{i+1}}_{i+1}}
=Z_{M}\bigl(z_{\bk_{[i]}} \sh z_{\overline{\bk^{[i]}+\bl}}\bigr)
\]
by Lemma \ref{lem:partial fraction shuffle}. 
This shows the formula \eqref{Taylor exp} for $\bullet=\sh$. 

For any admissible index $\bh$, we have $\lim\limits_{M\rightarrow \infty}Z_M(z_{\bh})=Z(z_{\bh})=Z^{\bullet}(z_{\bh})$. 
Thus we also have 
$\lim\limits_{M\rightarrow \infty}Z_M\bigl(w^{\bullet}_{\wcS, n}(\bk)\bigr)
=Z^{\bullet}\bigl(w^{\bullet}_{\wcS, n}(\bk)\bigr)$, 
since $w^{\bullet}_{\wcS, n}(\bk)\in\frH^0$ by Proposition~\ref{w in H0}. 
Therefore, by \eqref{Taylor exp}, we obtain 
\begin{equation}\label{lim exp of zeta_M}
\lim_{M\rightarrow \infty}\zeta^{\bullet}_{\wcS, M}(\bk)
=\sum_{n=0}^{\infty}Z^{\bullet}\bigl(w^{\bullet}_{\wcS, n}(\bk)\bigr)t^n.
\end{equation}

Finally, by using the fact that $Z^\bullet$ is a $\bQ$-algebra homomorphism with respect to the product $\bullet$, 
we calculate the right hand side of \eqref{lim exp of zeta_M} as
\begin{align*}
\sum_{n=0}^\infty Z^\bullet \bigl(w^\bullet_{\wcS, n}(\bk)\bigr)t^n
&=\sum_{n=0}^\infty Z^\bullet \Biggl(\sum_{i=0}^r(-1)^{\wt(\bk^{[i]})}
z_{\bk_{[i]}}\bullet \sum_{\substack{\bl\in\bZ_{\geq0}^{r-i} \\ \wt(\bl)=n}}
b\binom{\bk^{[i]}}{\bl}z_{\overline{\bk^{[i]}+\bl}}\Biggr)t^n\\
&=\sum_{n=0}^\infty \sum_{i=0}^r(-1)^{\wt(\bk^{[i]})}
Z^\bullet(z_{\bk_{[i]}})\sum_{\substack{\bl\in\bZ_{\geq0}^{r-i} \\ \wt(\bl)=n}}
b\binom{\bk^{[i]}}{\bl}Z^\bullet(z_{\overline{\bk^{[i]}+\bl}})t^n\\
&=\sum_{n=0}^\infty \sum_{i=0}^r(-1)^{\wt(\bk^{[i]})}
\zeta^\bullet(\bk_{[i]})\sum_{\substack{\bl\in\bZ_{\geq0}^{r-i} \\ \wt(\bl)=n}}
b\binom{\bk^{[i]}}{\bl}\zeta^\bullet(\overline{\bk^{[i]}+\bl})t^n. 
\end{align*}
This is exactly the definition of $\zeta^\bullet_{\wcS}(\bk)$ (Definition~\ref{def:t-adic SMZV}). 
Now the proof is complete. 
\end{proof}
By taking the constant term of $\zeta^\bullet_{\wcS}(\bk)$, 
we obtain a series expression of $\zeta^{\bullet}_\cS(\bk)$ as follows.
\begin{corollary}\label{cor:limits}
For an index $\bk$ and $\bullet\in\{\ast,\sh\}$, we have
\[
\lim_{M\rightarrow \infty}\zeta^\bullet_{\cS, M}(\bk)=\zeta^{\bullet}_{\cS}(\bk),
\]
where
\begin{align*}
&\zeta^{\ast}_{\cS, M}(\bk)
\coloneqq\sum_{i=0}^r\sum_{\substack{0<n_1<\cdots<n_i<M \\ -M<n_{i+1}<\cdots<n_r<0}}
\frac{1}{n^{k_1}_1\cdots n^{k_r}_r} \in \bQ,\\
&\zeta^{\sh}_{\cS, M}(\bk)
\coloneqq\sum_{i=0}^r\sum_{\substack{0<n_1<\cdots<n_i \\ n_{i+1}<\cdots<n_r<0 \\ n_i-n_{i+1}<M}}
\frac{1}{n^{k_1}_1\cdots n^{k_r}_r} \in \bQ.
\end{align*}
\end{corollary}
\subsection{The harmonic relation for ($*$-truncated) $\wcS$-MZVs}\label{subsec:harmonic rel}
In this subsection, we prove the identity \eqref{truncated harmonic rel}, i.e., 
the harmonic relation for $*$-truncated $\wcS$-MZVs. 
Then we obtain the same relation for $\wcS$-MZVs by taking the limit. 

\begin{theorem}\label{har for GSMZVSM*}
For any indices $\bk, \bl$, we have
\begin{equation}\label{har for M}
\zeta^{\ast}_{\wcS, M}(\bk\ast\bl)=\zeta^{\ast}_{\wcS, M}(\bk)\zeta^{\ast}_{\wcS, M}(\bl).
\end{equation}
\end{theorem}
\begin{proof}
For a non-zero integer $n$, set
\begin{equation}\label{n(t)}
n(t)\coloneqq
\begin{cases}
n & (n>0), \\
n+t & (n<0).
\end{cases}
\end{equation}
Then, we see that
\[
\zeta^{\ast}_{\wcS, M}(\bk)=\sum_{\substack{n_1\prec \cdots \prec n_r \\ 0<|n_1|, \ldots, |n_r|<M}}\frac{1}{n_1(t)^{k_1}\cdots n_r(t)^{k_r}}, 
\]
where $\prec$ is Kontsevich's order on the set 
$(\bZ\setminus\{0\})\cup\{\infty=-\infty\}$ defined as
\[
1 \prec 2 \prec \cdots \prec \infty=-\infty \prec \cdots \prec -2 \prec -1.
\]
This is a natural generalization of \cite[(9.1)]{Kan}.
By this expression, we can prove the harmonic relation for $\ast$-truncated $\wcS$-MZVs in exactly the same way 
as the harmonic relation for truncated MZVs.
\end{proof}

\begin{corollary}\label{har for GSMZVS*}
For any indices $\bk$ and $\bl$, we have
\[
\zeta^{\ast}_{\wcS}(\bk*\bl)=\zeta^{\ast}_{\wcS}(\bk)\zeta^{\ast}_{\wcS}(\bl).
\]
\end{corollary}
\begin{proof}
Take the limit $M\to\infty$ in \eqref{har for M} and use Theorem~\ref{lim}. 
\end{proof}
\section{Truncated $\wcS$-MZV associated with a $2$-colored rooted tree}\label{sec3}
\subsection{Definition and example}\label{subsec3.1}
In this subsection, we define the \emph{truncated $\wcS$-MZVs 
associated with $2$-colored rooted trees}. When the tree is linear, 
we recover the $\sh$-truncated $\wcS$-MZV $\zeta^{\sh}_{\wcS, M}(\bk)$ defined in Definition~\ref{def:truncated-sums}.

First, we recall the definition of $2$-colored rooted trees.
\begin{definition}[$2$-colored rooted tree {\cite[Definition~1.2]{O}}]\label{def of 2-crt}
A \emph{$2$-colored rooted tree} is a quadruple $X=(V, E, \rt, V_\bullet)$ consisting of the following data:
\begin{enumerate}
\item $(V, E)$ is a finite tree with the set of vertices $V$ and the set of edges $E$. 
Note that $\#V=\#E+1<+\infty$. 
\item $\rt \in V$ is a vertex, called the \emph{root}.
\item $V_\bullet$ is a subset of $V$ containing all terminals of $(V, E)$. 
Here, a \emph{terminal vertex} means a vertex of degree $1$.
\end{enumerate}
\end{definition}
We call a tuple $(k_e)_{e\in E}\in \bZ^{E}_{\geq0}$ an \emph{index} on $X$.
Recall that $t$ denotes an indeterminate. 
\begin{definition}\label{def of GSMZV assoc. with 2-crt}
Let $X=(V, E, \rt, V_\bullet)$ be a $2$-colored rooted tree, $u$ an element of $V_\bullet$ and $\bk=(k_e)_{e\in E} \in \bZ^{E}_{\geq0}$ an index on $X$.
Then we define 
\[
\zeta^{}_M(X, u; \bk)\coloneqq\sum_{(m_v)_{v\in V_{\bullet}} \in I_M(V_\bullet,u)}
\prod_{e \in E}\left(\sum_{v \in V_\bullet \text{ s.t.\ } e \in P(\rt, v)}
(m_v+\delta_{u,v}t)\right)^{-k_e}\in\mathbb{Q}\jump{t}.
\]
Here $P(\rt, v)$ is the path from the root $\rt$ to $v$, $\delta_{u,v}$ is the Kronecker delta and
\[
I_M(V_{\bullet},u)\coloneqq
\biggl\{(m_v)_{v\in V_{\bullet}} \in \bZ^{V_{\bullet}} \biggm| 
m_v>0 \; (v\neq u),\ -M<m_{u}<0,\ \sum_{v\in V_{\bullet}}m_v=0\biggr\}.
\]
Moreover, we define the \emph{truncated $\wcS$-MZV associated with $X$} by
\[
\zeta^{}_{\wcS, M}(X; \bk)\coloneqq\sum_{u\in V_\bullet}\zeta^{}_M(X, u; \bk). 
\]
\end{definition}

In the following, we denote by $L_e\bigl(X,u; (m_v)_{v\in V_{\bullet}})$ 
the factor appearing in the definition of $\zeta^{}_{M}(X,u; \bk)$, that is,
\[
L_e\bigl(X,u; (m_v)_{v\in V_{\bullet}}\bigr)\coloneqq
\sum_{v \in V_\bullet \text{ s.t.\ } e \in P(\rt, v)}(m_v+\delta_{u,v}t).
\]

We use diagrams to indicate $2$-colored rooted trees. 
For $X=(V, E, \rt, V_\bullet)$, 
the symbol $\bullet$ (resp.~$\circ$) denotes a vertex in $V_{\bullet}$ 
(resp.~$V_\circ\coloneqq V\setminus V_\bullet$) which is not the root. 
We use the symbol $\blacksquare$ or $\square$ to denote the root 
according to whether the root belongs to $V_\bullet$ or not. 
We also use the symbol $\times$ as the wild-card, i.e., 
to indicate a vertex which may or may not belong to $V_\bullet$ and 
may or may not be the root. 
If endpoints of an edge $e$ are $v$ and $v'$, then we will sometimes express $e$ by the set $\{v,v'\}$.

\begin{example}\label{ex_of_2crt1}
For an integer $r\geq 0$, let us consider the linear tree with $r+1$ vertices $v_1,\ldots,v_{r+1}$ 
and $r$ edges $e_a\coloneqq\{v_a,v_{a+1}\}$ ($a=1,\ldots,r$). 
We define the $2$-colored rooted tree $X=(V, E, \rt, V_\bullet)$ 
by setting $\rt\coloneqq v_{r+1}$ and $V_\bullet\coloneqq V=\{v_1,\ldots,v_{r+1}\}$. 
We identify an index $\bk=(k_1,\ldots,k_r)$ in the usual sense 
with an index on $X$ by setting $k_a\coloneqq k_{e_a}$. 
This situation is indicated by the diagram: 
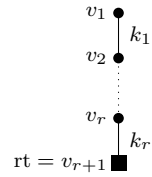
\[
\begin{tikzpicture}
\coordinate (v_1) at (0,2.0) node at (v_1) [left] {\tiny $v_1$};
\coordinate (v_2) at (0,1.4) node at (v_2) [left] {\tiny $v_2$};
\coordinate (v_r) at (0,0.6) node at (v_r) [left] {\tiny $v_r$};
\coordinate (v_{r+1}) at (0,0) node at (v_{r+1}) [left] {\tiny $\rt=v_{r+1}$};
\draw  (0,1.4) -- node [right] {\tiny $k_1$} (0,2.0);
\draw [dotted] (0,0.6) -- (0,1.4);
\draw  (0,0) -- node [right] {\tiny $k_r$} (0,0.6);
\fill (v_1) circle (2pt) (v_2) circle (2pt) (v_r) circle (2pt);
\fill (-0.1,-0.1) rectangle (0.1,0.1);
\end{tikzpicture}
\]
For $(m_v)_{v\in V_\bullet}\in\bZ^{V_\bullet}$, 
we put $m_i\coloneqq m_{v_i}$ ($i=1,\ldots,r+1$) and 
\begin{equation}\label{eq:M_ij}
M_{i,j}\coloneqq m_i+\cdots+m_j,\qquad M_j\coloneqq M_{1,j}
\end{equation}
for $1\leq i\leq j\leq r+1$. Then we have 
 \[
L_{e_a}\bigl(X,v_i;(m_v)_{v\in V_{\bullet}}\bigr)=\begin{cases}
M_a & (a<i),\\
M_a+t & (a\geq i) 
\end{cases}
\]
for $1 \leq a \leq r$ and $1\leq i\leq r+1$.
Thus we obtain
\begin{align*}
&\zeta^{}_M(X, v_i; \bk)\\
&=\sum_{\substack{m_1, \ldots, m_{i-1}, m_{i+1}, \ldots, m_{r+1}>0 \\ -M<m_i<0 \\
\sum_{j=1}^{r+1}m_j=0}}
\frac{1}{M_1^{k_1}\cdots M_{i-1}^{k_{i-1}}(M_i+t)^{k_i} \cdots (M_r+t)^{k_r}}\\
&=\sum_{\substack{m_1, \ldots, m_{i-1}, m_{i+1}, \ldots, m_{r+1}>0 \\ M_{i-1}+M_{i+1, r+1} < M}}\frac{1}{M^{k_1}_1\cdots M^{k_{i-1}}_{i-1}(t-M_{i+1, r+1})^{k_i}\cdots(t-M_{r+1, r+1})^{k_r}}\\
&=\sum_{\substack{0<n_1<\cdots<n_{i-1} \\ n_i<\cdots<n_r<0 \\ n_{i-1}-n_i< M}}
\frac{1}{n^{k_1}_1\cdots n^{k_{i-1}}_{i-1}(n_i+t)^{k_i}\cdots (n_r+t)^{k_r}}, 
\end{align*}
and hence 
\[
\zeta^{}_{\wcS, M}(X; \bk) = \sum_{i=1}^{r+1}\zeta^{}_M(X, v_i; \bk)=\zeta^{\sh}_{\wcS, M}(\bk). 
\]
Therefore, the truncated $\wcS$-MZV associated with a 2-colored rooted tree 
generalizes the $\sh$-truncated $\wcS$-MZV. 
\end{example}

\subsection{Basic properties}\label{subsec3.2}
In this subsection, we give basic properties for truncated $\wcS$-MZVs associated with $2$-colored rooted trees.
The proofs proceed in almost the same way as those in \cite{O}.
\begin{proposition}\label{contracting1}
Let $X=(V, E, \rt, V_\bullet)$ be a $2$-colored rooted tree and $\bk=(k_{e'})_{e'\in E}$ an index on $X$.
Assume that there exists an edge $e=\{a,b\} \in E$ such that $b\in V_\circ\setminus\{\rt\}$ and $k_e=0$. 
Let $(V', E')$ be the tree obtained by contracting the edge $e$, 
as represented by the following figure:
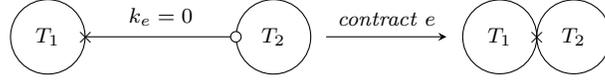
\[
\begin{tikzpicture}
\coordinate (T_1) at (-4.5,0) node at (T_1) {\tiny $T_1$};
\coordinate (T_2) at (-1.5,0) node at (T_2) {\tiny $T_2$};
\coordinate (a) at (-4,0) node at (a) {\tiny $\times$};
\coordinate (b) at (-2,0) node at (b) {}; 
\draw (-4,0) to node [above] {\tiny $k_e=0$} (-2,0);
\draw (T_1) circle [radius=14.1pt];
\draw (T_2) circle [radius=14.1pt];
\filldraw [fill=white] (b) circle [radius=0.7mm];
\draw[->,>=stealth] (-0.8,0) -- node [above] {\tiny contract $e$} (0.8,0);
\coordinate (T_1) at (1.508,0) node at (T_1) {\tiny $T_1$};
\coordinate (T_2) at (2.492,0) node at (T_2) {\tiny $T_2$};
\coordinate (a) at (2,0) node at (a) {\tiny $\times$};
\draw (T_1) circle [radius=14pt];
\draw (T_2) circle [radius=14pt];
\end{tikzpicture}
\]
Identifying $V'$ with $V\setminus\{b\}$ and $E'$ with $E\setminus\{e\}$, 
we define a $2$-colored rooted tree $X'\coloneqq(V', E', \rt, V_{\bullet})$ and 
an index $\bk'\coloneqq(k_{e'})_{e'\in E\setminus\{e\}}$ on $X'$. 
Then we have
\[
\zeta^{}_{\wcS, M}(X; \bk)=\zeta^{}_{\wcS, M}(X'; \bk').
\]
\end{proposition}
\begin{proof}
It follows from the equality 
\begin{align*}
&\zeta^{}_M(X,u; \bk)\\
&=\sum_{(m_v)_{v\in V_{\bullet}} \in I_M(V_{\bullet},u)}
L_e\bigl(X,u; (m_v)_{v\in V_{\bullet}}\bigr)^{-k_e}
\prod_{e'\in E\setminus\{e\}}L_{e'}\bigl(X,u; (m_v)_{v\in V_{\bullet}}\bigr)^{-k_{e'}}\\
&=\sum_{(m_v)_{v\in V_{\bullet}} \in I_M(V_{\bullet},u)}
\prod_{e'\in E'}L_{e'}\bigl(X',u; (m_v)_{v\in V_{\bullet}}\bigr)^{-k_{e'}}\\
&=\zeta^{}_M(X',u; \bk'), 
\end{align*}
which holds for any $u\in V_\bullet$. 
\end{proof}
\begin{proposition}\label{contracting2}
Let $X=(V, E, \rt, V_\bullet)$ be a $2$-colored rooted tree and $\bk=(k_e)_{e\in E}$ an index on $X$. 
Assume that there is a vertex $b\in V_\circ\setminus\{\rt\}$ of degree $2$, 
and let $e_1=\{a,b\}$ and $e_2=\{b,c\}$ be the edges incident on $b$. 
By setting $V'\coloneqq V\setminus\{b\}$, $e_{12}\coloneqq\{a,c\}$ and $E'\coloneqq (E\setminus\{e_1,e_2\})\cup\{e_{12}\}$, 
we define a $2$-colored rooted tree $X'\coloneqq(V', E', \rt, V_{\bullet})$. 
Moreover, we define an index $\bk'=(k'_e)_{e\in E'}$ on $X'$ by putting $k'_{e_{12}}\coloneqq k_{e_1}+k_{e_2}$ 
and $k'_e\coloneqq k_e$ for other edges $e\in E'\setminus\{e_{12}\}$. 
This situation can be represented by the following figure:
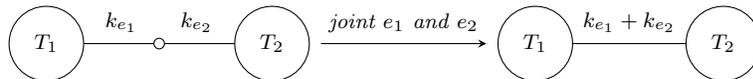
\[
\begin{tikzpicture}
\coordinate (T_1) at (-4.5,0) node at (T_1) {\tiny $T_1$};
\coordinate (T_2) at (-1.5,0) node at (T_2) {\tiny $T_2$};
\coordinate (b) at (-3,0) node at (b) {}; 
\draw (-4,0) to node [above] {\tiny $k_{e_1}$} (b);
\draw (b) to node [above] {\tiny $k_{e_2}$} (-2,0);
\draw (T_1) circle [radius=14.1pt];
\draw (T_2) circle [radius=14.1pt];
\filldraw [fill=white] (b) circle [radius=0.7mm];
\draw[->,>=stealth] (-0.85,0) -- node [above] {\tiny joint $e_1$ and $e_2$} (1.35,0);
\coordinate (T_1) at (2,0) node at (T_1) {\tiny $T_1$};
\coordinate (T_2) at (4.5,0) node at (T_2) {\tiny $T_2$};
\draw (2.5,0) to node [above] {\tiny $k_{e_1}+k_{e_2}$} (4,0);
\draw (T_1) circle [radius=14.1pt];
\draw (T_2) circle [radius=14.1pt];
\end{tikzpicture}
\]
Then we have
\[
\zeta^{}_{\wcS, M}(X; \bk)=\zeta^{}_{\wcS, M}(X'; \bk').
\]
\end{proposition}
\begin{proof}
Since $b$ is a vertex in $V_\circ\setminus\{\rt\}$, we have
\[
\{v \in V_\bullet \mid e_1\in P(\rt, v)\}=\{v \in V_\bullet \mid e_2 \in P(\rt, v)\}.
\]
Therefore, for $u \in V_\bullet$, we obtain
\begin{align*}
&\zeta^{}_M(X,u; \bk)\\
&=\sum_{(m_v)_{v\in V_{\bullet}} \in I_M(V_{\bullet},u)}
\prod_{e\in E} L_e\bigl(X,u;(m_v)_{v\in V_{\bullet}}\bigr)^{-k_e}\\
&=\sum_{(m_v)_{v\in V_{\bullet}} \in I_M(V_{\bullet},u)}
L_{e_1}\bigl(X,u; (m_v)_{v\in V_{\bullet}}\bigr)^{-(k_{e_1}+k_{e_2})}
\prod_{e \in E \setminus\{e_1, e_2\}} L_{e}\bigl(X,u;(m_v)_{v\in V_{\bullet}}\bigr)^{-k_{e}}\\
&=\sum_{(m_v)_{v\in V_{\bullet}} \in I_M(V_{\bullet},u)}
L_{e_{12}}\bigl(X',u; (m_v)_{v\in V_{\bullet}}\bigr)^{-k'_{e_{12}}}
\prod_{e \in E' \setminus\{e_{12}\}} L_{e}\bigl(X',u;(m_v)_{v\in V_{\bullet}}\bigr)^{-k'_{e}}\\
&=\zeta^{}_M(X', u; \bk').
\end{align*}
Thus we complete the proof by taking the sum over all $u\in V_\bullet$.
\end{proof}
\begin{proposition}\label{change of root}
Let $X_1$ and $X_2$ be $2$-colored rooted trees which are distinct only in their roots, 
written as $X_i=(V, E, v_i, V_\bullet)$ ($i=1,2$). 
Then, for any index $\bk=(k_e)_{e\in E}$ on $X_1$, we have
\begin{multline}\label{eq:change of root}
\zeta^{}_{\wcS, M}(X_1; \bk)
=(-1)^{\sum_{e \in P(v_1, v_2)}k_e}\sum_{\bl=(l_e)\in\bZ_{\geq0}^{P(v_1,v_2)}}
\left[\prod_{e\in P(v_1, v_2)}\binom{k_e+l_e-1}{l_e}\right]\\
\times\zeta^{}_{\wcS, M}(X_2; \bk\oplus\bl)\,t^{\sum_{e \in P(v_1, v_2)}l_e}. 
\end{multline}
Here $\bk\oplus\bl$ denotes an index on $X_2$ whose $e$-component is given by 
\[
\begin{cases}
k_e+l_e & (e \in P(v_1, v_2)), \\
k_e & (e \not\in P(v_1, v_2)), 
\end{cases}
\]
and we use the convention
\[
\binom{l-1}{l}=
\begin{cases}
1 & (l=0), \\
0 & (l>0).
\end{cases}
\]
\end{proposition}
\begin{proof}
Take $u\in V_{\bullet}$ and $e\in E$.
First we consider the case $e \in P(v_1, v_2)$.
This situation can be illustrated as follows: 
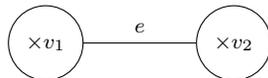
\[
\begin{tikzpicture}
\coordinate (T_1) at (2,0) node at (T_1) {\tiny $\times v_1$};
\coordinate (T_2) at (4.5,0) node at (T_2) {\tiny $\times v_2$};
\draw (2.5,0) to node [above] {\tiny $e$} (4,0);
\draw (T_1) circle [radius=14.1pt];
\draw (T_2) circle [radius=14.1pt];
\end{tikzpicture}
\]
Since
\begin{equation}\label{partitioned}
V_\bullet=\{v\in V_\bullet \mid e \in P(v_1, v)\}\sqcup\{v\in V_\bullet \mid e \in P(v_2, v)\},
\end{equation}
we have
\begin{align*}
&L_e\bigl(X_1, u; (m_v)_{v\in V_{\bullet}}\bigr)+L_e\bigl(X_2, u; (m_v)_{v\in V_{\bullet}}\bigr)\\
&=\sum_{v \in V_\bullet \text{ s.t. } e \in P(v_1, v)}(m_v+\delta_{u,v}t)
+\sum_{v \in V_\bullet \text{ s.t. } e \in P(v_2, v)}(m_v+\delta_{u,v}t)\\
&=\sum_{v \in V_\bullet}m_v+\sum_{v\in V_\bullet}\delta_{u,v}t=t.
\end{align*}

Next we consider the case $e\not\in P(v_1, v_2)$, which can be illustrated as follows:
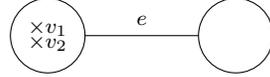
\[
\begin{tikzpicture}
\coordinate (T_1) at (2,0) node at (T_1) {};
\coordinate (T_2) at (4.5,0) node at (T_2) {};
\coordinate (v_1) at (2,0.1) node at (v_1) {\tiny $\times v_1$};
\coordinate (v_2) at (2,-0.1) node at (v_2) {\tiny $\times v_2$};
\draw (2.5,0) to node [above] {\tiny $e$} (4,0);
\draw (T_1) circle [radius=14.1pt];
\draw (T_2) circle [radius=14.1pt];
\end{tikzpicture}
\]
Since
\[
\{v \in V_\bullet \mid e \in P(v_1, v)\}=\{v \in V_\bullet \mid e \in P(v_2, v)\},
\]
we obtain
\begin{align*}
L_e\bigl(X_1,u; (m_v)_{v\in V_{\bullet}}\bigr)
&=\sum_{v \in V_\bullet \text{ s.t. } e\in P(v_1, v)}(m_v+\delta_{u,v}t)\\
&=\sum_{v \in V_\bullet \text{ s.t. } e\in P(v_2, v)}(m_v+\delta_{u,v}t)
=L_e\bigl(X_2,u; (m_v)_{v\in V_{\bullet}}\bigr).
\end{align*}
Thus we have
\begin{multline*}
\zeta^{}_{ M}(X_1,u; \bk)
=\sum_{(m_v)_{v\in V_{\bullet}} \in I_M(V_{\bullet},u)}
\prod_{e \in P(v_1, v_2)}\Bigl(t-L_e\bigl(X_2, u; (m_v)_{v\in V_{\bullet}}\bigr)\Bigr)^{-k_e}\\
\times \prod_{e\in E\setminus P(v_1, v_2)}L_e\bigl(X_2, u; (m_v)_{v\in V_{\bullet}}\bigr)^{-k_e}.
\end{multline*}
Therefore, we obtain the desired formula by expanding the factors $(t-L_e)^{-k_e}$ as 
\[
\frac{1}{(t-L_e)^{k_e}}=(-1)^{k_e}\sum_{l_e=0}^{\infty}\binom{k_e+l_e-1}{l_e}\frac{t^{l_e}}{L^{k_e+l_e}_e}, 
\]
and taking the sum over $u \in V_\bullet$. 
\end{proof}

\begin{lemma}\label{lem:one-step}
Consider a $2$-colored rooted tree $X=(V, E, \rt, V_\bullet)$ and 
an index $\bk=(k_e)_{e\in E}$ on $X$ of the following shape:  
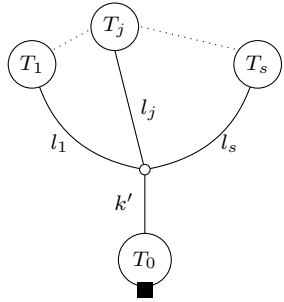
\[
\begin{tikzpicture}
\coordinate (T_0) at (0,0.4) node at (0,0.4) {\tiny $T_0$};
\coordinate (T_1) at (-1.5,3) node at (T_1) {\tiny $T_1$};
\coordinate (T_j) at (-0.4,3.5) node at (T_j) {\tiny $T_j$};
\coordinate (T_s) at (1.5,3) node at (T_s) {\tiny $T_s$};
\coordinate (v_0) at (0,1.6) node at (0,1.6) {};
\coordinate (v_1) at (0,0.75) node at (0,0.75) {};
\coordinate (v_{r+1}) at (0,0.0) node at (0,0) {};
\draw              (-1.4,2.7) to [out=290,in=170] node [left] {\tiny $l_1$} (v_0);
\draw [dotted] (-1.21,3.2) -- (-0.75,3.45);
\draw              (-0.4,3.19) -- node [right] {\tiny $l_j$} (v_0);
\draw [dotted] (-0.05,3.5) -- (1.21,3.2);
\draw              (1.4,2.7) to [out=250, in=10] node [right] {\tiny $l_s$} (v_0);
\draw              (v_0) -- node [left] {\tiny $k'$} (v_1);
\draw (T_1) circle [radius=9pt];
\draw (T_j) circle [radius=9pt];
\draw (T_s) circle [radius=9pt];
\draw (T_0) circle [radius=10pt];
\fill (-0.1,-0.1) rectangle (0.1,0.1);
\filldraw[fill=white] (v_0) circle [radius=0.7mm];
\end{tikzpicture}
\]
Here, $s$ and  $l_i \coloneqq k_{e_i}$ ($1\le i \le s$) are positive integers and 
$k' \coloneqq k_{e'}$ is a non-negative integer, 
where $\{e_i\}$ and $e'$ are the corresponding edges in $E$. $T_0, \ldots, T_s$ are subtrees of $(V, E)$.
Moreover, for $1\le i \le s$, let $\bh_i$ be the index on $X$ whose $e$-component is
\[
\begin{cases}
l_i-1 & (e=e_i),\\
k'+1 & (e=e'),\\
k_e & (\text{otherwise}).
\end{cases}
\]
Then, we have
\begin{equation}\label{eq:one-step}
\zeta^{}_{\wcS,M}(X;\bk)=\sum_{i=1}^s\zeta^{}_{\wcS,M}(X;\bh_i).
\end{equation}
\end{lemma}

\begin{proof}
For a tree $T$, we denote by $V(T)$ (resp.\ $E(T)$) the set of vertices (resp.\ edges) of $T$.
By definition, we have
\begin{multline*}
\zeta^{}_{\wcS,M}(X;\bk)=\sum_{u\in V_\bullet}\sum_{(m_v)_{v\in V_{\bullet}}\in I_M(V_{\bullet},u)}
\prod_{i=1}^s\frac{1}{L_{e_i}\bigl(X,u;(m_v)_{v\in V_{\bullet}}\bigr)^{k_{e_i}}}\\
\times\frac{1}{L_{e'}\bigl(X,u;(m_v)_{v\in V_{\bullet}}\bigr)^{k'}}
\prod_{i=0}^s\prod_{e \in E(T_i)}\frac{1}{L_{e}\bigl(X,u;(m_v)_{v\in V_{\bullet}}\bigr)^{k_e}}.
\end{multline*}
If we abbreviate as $L_{e'}=L_{e'}\bigl(X,u;(m_v)_{v\in V_{\bullet}}\bigr)$ 
and $L_i=L_{e_i}\bigl(X,u;(m_v)_{v\in V_{\bullet}}\bigr)$ ($i=1,\ldots,s$), 
we have $L_{e'}=L_1+\cdots+L_s$ by definition, and hence 
\[\frac{1}{L_1\cdots L_s}=\frac{1}{L_{e'}}\sum_{i=1}^s\frac{L_i}{L_1\cdots L_s}. \]
This implies \eqref{eq:one-step}.
\end{proof}

We define a $\bQ$-linear map $Z^{\sh}_{\wcS,M}\colon\frH^1\to\bQ\jump{t}$ 
by $Z^{\sh}_{\wcS,M}(z_{\bk})\coloneqq \zeta^{\sh}_{\wcS,M}(\bk)$. 
We prove that a truncated $\wcS$-MZV associated with 
a 2-colored rooted tree of a certain shape can be written explicitly as a value of $Z^{\sh}_{\wcS, M}$. 

\begin{theorem}\label{thm: explicit calc}
Let $\bk_i=(k_{i,1},\ldots,k_{i,r_i})$ ($i=1,\ldots,s$) be $s$ ($\ge 1$) non-empty indices, 
$\bk'=(k'_1,\ldots,k'_r)$ be a (possibly empty) index, and $k'$ be a non-negative integer. 
Consider a $2$-colored rooted tree $X=(V, E, \rt, V_\bullet)$ and 
an index $\bk=(k_e)_{e\in E}$ on $X$ of the following shape:  
\[
\begin{tikzpicture}
\coordinate  (v_{1, 1}) at (-3, -5.3) node at (v_{1, 1}) [left] {\tiny $v_{1, 1}$};
\coordinate  (v_{1, 2}) at (-3, -6.0) node at (v_{1, 2}) [left] {\tiny $v_{1, 2}$};
\coordinate  (v_{1, r_1-1}) at (-3, -6.8) node at (v_{1, r_1-1}) [left] {\tiny $v_{1, r_1-1}$};
\coordinate  (v_{1, r_1}) at (-3, -7.5) node at (v_{1, r_1}) [left] {\tiny $v_{1, r_1}$};
\coordinate (v_{i,1}) at (0,-4.7) node at (v_{i,1}) [left] {\tiny $v_{i,1}$};
\coordinate (v_{i,2}) at (0,-5.4) node at (v_{i,2}) [left] {\tiny $v_{i,2}$};
\coordinate (v_{i,r_i-1}) at (0,-6.3) node at (v_{i,r_i-1}) [left] {\tiny $v_{i,r_i-1}$};
\coordinate (v_{i,r_i}) at (0,-7) node at (v_{i,r_i}) [left] {\tiny $v_{i,r_i}$};
\coordinate  (v_{s,1}) at (3, -5.3) node at (v_{s,1}) [right] {\tiny $v_{s, 1}$};
\coordinate  (v_{s,2}) at (3, -6.0) node at (v_{s,2}) [right] {\tiny $v_{s, 2}$};
\coordinate  (v_{s,r_s-1}) at (3, -6.8) node at (v_{s,r_s-1}) [right] {\tiny $v_{s, r_s-1}$};
\coordinate  (v_{s,r_s}) at (3, -7.5) node at (v_{s,r_s}) [right] {\tiny $v_{s, r_s}$};
\coordinate (w) at (0,-8) node at (w) [above left] {\tiny $w$};
\coordinate (v_1) at (0,-8.6) node at (v_1) [left] {\tiny $v_1$};
\coordinate (v_2) at (0,-9.2) node at (v_2) [left] {\tiny $v_2$};
\coordinate (v_r) at (0,-10.1) node at (v_r) [left] {\tiny $v_r$};
\coordinate (rt) at (0,-10.7) node at (rt) [left] {\tiny $\rt$};
\draw              (-3,-6.0) --node [right] {\tiny $k_{1, 1}$} (-3,-5.3);
\draw [dotted] (-3,-6.8) -- (-3,-6.0);
\draw              (-3,-7.5) --node [right] {\tiny $k_{1, r_1-1}$} (-3,-6.8);
\draw              (0,-8) -- node [below] {\tiny $k_{1, r_1}$} (-3,-7.5);
\draw              (0,-5.4) -- node [right] {\tiny $k_{i,1}$} (0,-4.7);
\draw [dotted] (0,-6.3) -- (0,-5.4);
\draw              (0,-7) -- node [right] {\tiny $k_{i, r_i-1}$} (0,-6.3);
\draw              (0,-8) -- node [right] {\tiny $k_{i, r_i}$} (0,-7);
\draw              (3,-6.0) --node [left] {\tiny $k_{s, 1}$} (3,-5.3);
\draw [dotted] (3,-6.8) -- (3,-6.0);
\draw              (3,-7.5) --node [left] {\tiny $k_{s, r_s-1}$} (3,-6.8);
\draw              (0,-8) -- node [below] {\tiny $k_{s, r_s}$} (3,-7.5);
\draw [dotted] (-2,-5.7) -- (-0.5,-5.7);
\draw [dotted] (2,-5.7) -- (0.5,-5.7);
\draw              (0,-9.2) -- node [right] {\tiny $k'_1$} (0,-8.6) --node [right] {\tiny $k'$} (0,-8);
\draw [dotted] (0,-10.1) -- (0,-9.2);
\draw              (0,-10.7) -- node [right] {\tiny $k'_r$} (0,-10.1);
\fill (v_{1, 1}) circle (2pt) (v_{1, 2}) circle (2pt) (v_{1, r_1-1}) circle (2pt) (v_{1, r_1}) circle (2pt);
\fill (v_{i,1}) circle (2pt) (v_{i,2}) circle (2pt) (v_{i,r_i-1}) circle (2pt) (v_{i,r_i}) circle (2pt);
\fill (v_{s,1}) circle (2pt) (v_{s,2}) circle (2pt) (v_{s,r_s-1}) circle (2pt) (v_{s,r_s}) circle (2pt);
\fill (v_1) circle (2pt) (v_2) circle (2pt) (v_r) circle (2pt);
\fill (-0.1,-10.8) rectangle (0.1,-10.6);
\filldraw[fill=white] (w) circle[radius=0.7mm];
\end{tikzpicture}
\]
Then we have
\begin{align}\label{eq: special case formula}
\zeta^{}_{\wcS, M}(X; \bk)= Z^{\sh}_{\wcS, M}\bigl((z_{\bk_1} \sh \cdots \sh z_{\bk_s})x^{k'}z_{\bk'}\bigr).
\end{align}
\end{theorem}
\begin{proof}
We prove this theorem by induction on $\ell\coloneqq \sum_{i=1}^s\sum_{j=1}^{r_i}k_{i,j}$. 
When $\ell=1$, the statement reduces to Example \ref{ex_of_2crt1} via Proposition \ref{contracting2}. 
Next we assume $\ell>1$ and that the theorem holds for the case of $\ell-1$. 
Denote by $e_i$ and $e'$ the edges to which the components $k_{i,r_i}$ and $k'$ of the index $\bk$ is attached, 
respectively. Then, from Lemma~\ref{lem:one-step}, we have 
\begin{equation}\label{eq:keyofMT}
\zeta^{}_{\wcS,M}(X;\bk)=\sum_{i=1}^s\zeta^{}_{\wcS,M}(X;\bh_i),
\end{equation}
where $\bh_i$ is the index on $X$ whose $e$-component is
\[
\begin{cases}
k_{i,r_i}-1 & (e=e_i),\\
k'+1 & (e=e'),\\
k_e & (\text{otherwise}).
\end{cases}
\]
Note that the case of $\bk'=\varnothing$ corresponds to the case of $T_0=\{\rt\}$ in Lemma \ref{lem:one-step}. 

We calculate $\zeta^{}_{\wcS,M}(X;\bh_i)$ for each $i$.
If $k_{i,r_i}\geq 2$, we have
\begin{equation}\label{eq:ind1}
\zeta^{}_{\wcS,M}(X;\bh_i)=Z^{\sh}_{\wcS,M}\left(\biggl(\foo_{\substack{a=1 \\ a\neq i}}^sz_{\bk_a}\sh z_{k_{i,1}}\cdots z_{k_{i, r_i-1}}z_{k_{i,r_i}-1}\biggr)x^{k'+1}z_{\bk'}\right)
\end{equation}
by the induction hypothesis. 
Next we consider the case $k_{i,r_i}=1$. In this case, we construct 
a $2$-colored rooted tree $X'$ and an index $\bh'_i$ as indicated in the following diagram: 
\[
\begin{tikzpicture}
\coordinate  (v_{1,1}) at (-3, -5.3) node at (v_{1,1}) [left] {\tiny $v_{1,1}$};
\coordinate  (v_{1,2}) at (-3, -6.0) node at (v_{1,2}) [left] {\tiny $v_{1,2}$};
\coordinate  (v_{1,r_1-1}) at (-3, -6.8) node at (v_{1,r_1-1}) [left] {\tiny $v_{1,r_1-1}$};
\coordinate  (v_{1,r_1}) at (-3, -7.5) node at (v_{1,r_1}) [left] {\tiny $v_{1,r_1}$};
\coordinate (v_{i,1}) at (0,-4.7) node at (v_{i,1}) [left] {\tiny $v_{i,1}$};
\coordinate (v_{i,2}) at (0,-5.4) node at (v_{i,2}) [left] {\tiny $v_{i,2}$};
\coordinate (v_{i,r_i-2}) at (0,-6.3) node at (v_{i,r_i-2}) [left] {\tiny $v_{i,r_i-2}$};
\coordinate (v_{i,r_i-1}) at (0,-7) node at (v_{i,r_i-1}) [left] {\tiny $v_{i,r_i-1}$};
\coordinate  (v_{s,1}) at (3, -5.3) node at (v_{s,1}) [right] {\tiny $v_{s, 1}$};
\coordinate  (v_{s,2}) at (3, -6.0) node at (v_{s,2}) [right] {\tiny $v_{s, 2}$};
\coordinate  (v_{s,r_s-1}) at (3, -6.8) node at (v_{s,r_s-1}) [right] {\tiny $v_{s, r_s-1}$};
\coordinate  (v_{s,r_s}) at (3, -7.5) node at (v_{s,r_s}) [right] {\tiny $v_{s, r_s}$};
\coordinate (w) at (0,-8) node at (w) [above left] {\tiny $w$};
\coordinate (v_0) at (0,-8.6) node at (v_0) [left] {\tiny $v_0$};
\coordinate (v_1) at (0,-9.2) node at (v_1) [left] {\tiny $v_1$};
\coordinate (v_2) at (0,-9.8) node at (v_2) [left] {\tiny $v_2$};
\coordinate (v_r) at (0,-10.7) node at (v_r) [left] {\tiny $v_r$};
\coordinate (rt) at (0,-11.3) node at (rt) [left] {\tiny $\rt$};
\draw              (-3,-6.0) --node [right] {\tiny $k_{1, 1}$} (-3,-5.3);
\draw [dotted] (-3,-6.8) -- (-3,-6.0);
\draw              (-3,-7.5) --node [right] {\tiny $k_{1, r_1-1}$} (-3,-6.8);
\draw              (0,-8) -- node [below] {\tiny $k_{1, r_1}$} (-3,-7.5);
\draw              (0,-5.4) -- node [right] {\tiny $k_{i,1}$} (0,-4.7);
\draw [dotted] (0,-6.3) -- (0,-5.4);
\draw              (0,-7) -- node [right] {\tiny $k_{i, r_i-2}$} (0,-6.3);
\draw              (0,-8) -- node [right] {\tiny $k_{i, r_i-1}$} (0,-7);
\draw              (3,-6.0) --node [left] {\tiny $k_{s, 1}$} (3,-5.3);
\draw [dotted] (3,-6.8) -- (3,-6.0);
\draw              (3,-7.5) --node [left] {\tiny $k_{s, r_s-1}$} (3,-6.8);
\draw              (0,-8) -- node [below] {\tiny $k_{s, r_s}$} (3,-7.5);
\draw [dotted] (-2,-5.7) -- (-0.5,-5.7);
\draw [dotted] (2,-5.7) -- (0.5,-5.7);
\draw              (0,-9.8) -- node [right] {\tiny $k'_1$} (0,-9.2) -- node [right] {\tiny $k'+1$} (0,-8.6) --node [right] {\tiny $0$} (0,-8);
\draw [dotted] (0,-10.7) -- (0,-9.8);
\draw              (0,-11.3) -- node [right] {\tiny $k'_r$} (0,-10.7);
\fill (v_{1,1}) circle (2pt) (v_{1,2}) circle (2pt) (v_{1,r_1-1}) circle (2pt) (v_{1,r_1}) circle (2pt);
\fill (v_{i,1}) circle (2pt) (v_{i,2}) circle (2pt) (v_{i,r_i-2}) circle (2pt) (v_{i,r_i-1}) circle (2pt);
\fill (v_{s,1}) circle (2pt) (v_{s,2}) circle (2pt) (v_{s,r_s-1}) circle (2pt) (v_{s,r_s}) circle (2pt);
\fill (v_0) circle (2pt) (v_1) circle (2pt) (v_2) circle (2pt) (v_r) circle (2pt);
\fill (-0.1,-11.4) rectangle (0.1,-11.2);
\filldraw[fill=white] (w) circle[radius=0.7mm];
\end{tikzpicture}
\]
Here the edge $e_i$ of $X$ is contracted, and a new vertex $v_0$ is inserted between $w$ and $v_1$. 
The components of the index $\bh'_i$ on the edges $\{w,v_0\}$ and $\{v_0,v_1\}$ are 
set to be zero and $k'+1$, respectively. 
Then, according to Proposition~\ref{contracting1} and the induction hypothesis, we have 
\begin{equation}\label{eq:ind3}
\begin{split}
\zeta^{}_{\wcS,M}(X;\bh_i)&=\zeta^{}_{\wcS,M}(X';\bh'_i)\\
&=Z^{\sh}_{\wcS,M}\left(\biggl(\foo_{\substack{a=1 \\ a\neq i}}^s
z_{\bk_a}\sh z_{k_{i,1}}\cdots z_{k_{i, r_i-1}}\biggr)x^0z_{k'+1}z_{\bk'}\right)\\
&=Z^{\sh}_{\wcS,M}\left(\biggl(\foo_{\substack{a=1 \\ a\neq i}}^s
z_{\bk_a}\sh z_{k_{i,1}}\cdots z_{k_{i,r_i-1}}\biggr)yx^{k'}z_{\bk'}\right).
\end{split}
\end{equation}
By combining \eqref{eq:keyofMT}, \eqref{eq:ind1}, \eqref{eq:ind3} and the definition of the shuffle product, 
we have \eqref{eq: special case formula}.
\end{proof}

\begin{remark}
The value $\zeta^{}_{\wcS,M}$ which appears in Theorem \ref{thm: explicit calc} may be regarded as 
an analogue of MZVs studied by Umezawa \cite{U}. 
\end{remark}

\subsection{The shuffle relation for ($\sh$-truncated) $\wcS$-MZVs.}\label{subsec3.3}
In this subsection, we give a proof of the identity \eqref{truncated shuffle rel}, i.e., 
the shuffle relation for $\sh$-truncated $\wcS$-MZVs as an application of the basic properties in the previous subsection.

By specializing at $t=0$, this also gives a new proof of the shuffle relation for $\cS$-MZVs, 
which was proved by Kaneko--Zagier \cite{KZ}, Jarossay \cite[Th\'{e}or\`{e}me 1.7 i)]{J1} 
and Hirose \cite[Proposition 15]{Hi}. In fact, this is completely parallel to 
the proof of the shuffle relation for $\cA$-MZVs given by the first author \cite{O}. 
\begin{theorem}\label{thm:truncated shuffle}
For indices $\bk$, $\bl$, we have 
\begin{equation}\label{sh for M}
\zeta^{\sh}_{\wcS,M}(\bk \sh \bl)=(-1)^{\wt(\bl)}\sum_{\bl'\in \bZ^{\dep(\bl)}_{\geq0}}b\binom{\bl}{\bl'}\zeta^{\sh}_{\wcS,M}(\bk, \overline{\bl+\bl'}){t}^{\wt(\bl')}.
\end{equation}
\end{theorem}
\begin{proof}
We write $\bk=(k_1, \ldots, k_r)$, $\bl=(l_1, \ldots, l_s)$, 
and consider two $2$-colored rooted trees $X_1$ and $X_2$ and the index $\bk'=(k'_e)$ on them 
defined as in the following diagrams: 
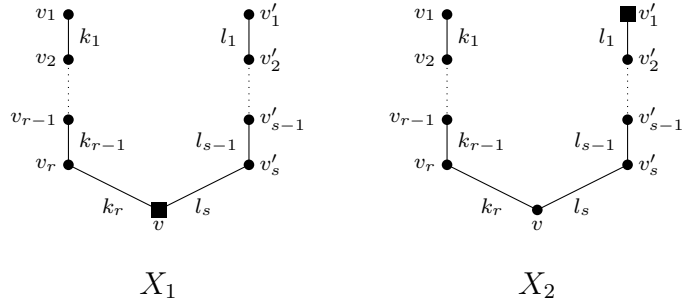
\[
\begin{tikzpicture}
\coordinate (v_1) at (-1.2, 2.6) node at (v_1) [left] {\tiny $v_1$};
\coordinate (v_2) at (-1.2,2.0) node at (v_2) [left] {\tiny $v_2$};
\coordinate (v_{r-1}) at (-1.2,1.2) node at (v_{r-1}) [left] {\tiny $v_{r-1}$};
\coordinate (v_r) at (-1.2,0.6) node at (v_r) [left] {\tiny $v_r$};
\coordinate (v'_1) at (1.2,2.6) node at (v'_1) [right] {\tiny $v'_1$};
\coordinate (v'_2) at (1.2,2.0) node at (v'_2) [right] {\tiny $v'_2$};
\coordinate (v'_{s-1}) at (1.2,1.2) node at (v'_{s-1}) [right] {\tiny $v'_{s-1}$};
\coordinate (v'_s) at (1.2,0.6) node at (v'_s) [right] {\tiny $v'_s$};
\coordinate (v) node at (0,0) [below] {\tiny $v$};
\draw              (-1.2,2.0) -- node [right] {\tiny $k_1$} (-1.2,2.6);
\draw [dotted] (-1.2,1.2) -- (-1.2,2.0);
\draw              (-1.2,0.6) -- node [right] {\tiny $k_{r-1}$} (-1.2,1.2);
\draw              (0,0) -- node [below] {\tiny $k_r$} (-1.2,0.6);
\draw              (1.2,2.0) -- node [left] {\tiny $l_1$} (1.2,2.6);
\draw [dotted] (1.2,1.2) -- (1.2,2.0);
\draw              (1.2,0.6) -- node [left] {\tiny $l_{s-1}$} (1.2,1.2);
\draw              (0,0) -- node [below] {\tiny $l_s$} (1.2,0.6);
\fill (v_1) circle (2pt) (v_2) circle (2pt) (v_{r-1}) circle (2pt) (v_r) circle (2pt);
\fill (v'_1) circle (2pt) (v'_2) circle (2pt) (v'_{s-1}) circle (2pt) (v'_s) circle (2pt);
\fill (-0.1,-0.1) rectangle (0.1,0.1);
\draw node at (0,-1) {$X_1$};
\end{tikzpicture}
\qquad 
\begin{tikzpicture}
\coordinate (v_1) at (-1.2,2.6) node at (-1.2,2.6) [left] {\tiny $v_1$};
\coordinate (v_2) at (-1.2,2.0) node at (-1.2,2.0) [left] {\tiny $v_2$};
\coordinate (v_{r-1}) at (-1.2,1.2) node at (-1.2,1.2) [left] {\tiny $v_{r-1}$};
\coordinate (v_r) at (-1.2,0.6) node at (-1.2,0.6) [left] {\tiny $v_r$};
\coordinate (v'_1) at (1.2,2.6) node at (1.2,2.6) [right] {\tiny $v'_1$};
\coordinate (v'_2) at (1.2,2.0) node at (1.2,2.0) [right] {\tiny $v'_2$};
\coordinate (v'_{s-1}) at (1.2,1.2) node at (1.2,1.2) [right] {\tiny $v'_{s-1}$};
\coordinate (v'_s) at (1.2,0.6) node at (1.2,0.6) [right] {\tiny $v'_s$};
\coordinate (v) at (0,0) node at (0,0) [below] {\tiny $v$};
\draw              (-1.2,2.0) -- node [right] {\tiny $k_1$} (-1.2,2.6);
\draw [dotted] (-1.2,1.2) -- (-1.2,2.0);
\draw              (-1.2,0.6) -- node [right] {\tiny $k_{r-1}$} (-1.2,1.2);
\draw              (0,0) -- node [below] {\tiny $k_r$} (-1.2,0.6);
\draw              (1.2,2.0) -- node [left] {\tiny $l_1$} (1.2,2.6);
\draw [dotted] (1.2,1.2) -- (1.2,2.0);
\draw              (1.2,0.6) -- node [left] {\tiny $l_{s-1}$} (1.2,1.2);
\draw              (0,0) -- node [below] {\tiny $l_s$} (1.2,0.6);
\fill (v_1) circle (2pt) (v_2) circle (2pt) (v_{r-1}) circle (2pt) (v_r) circle (2pt);
\fill (1.1,2.5) rectangle (1.3,2.7);
\fill (v'_2) circle (2pt) (v'_{s-1}) circle (2pt) (v'_s) circle (2pt);
\fill (v) circle (2pt);
\draw node at (0,-1) {$X_2$};
\end{tikzpicture}
\]
Since they are distinct only in their roots, Proposition~\ref{change of root} shows 
\begin{multline}\label{eq:sh-app0}
\zeta^{}_{\wcS, M}(X_1; \bk')
=(-1)^{\sum_{e \in P(v, v'_1)}k'_e}\sum_{\bl'=(l'_e)\in\bZ_{\geq0}^{P(v, v'_1)}}
\left[\prod_{e\in P(v, v'_1)}\binom{k'_e+l'_e-1}{l'_e}\right]\\
\times\zeta^{}_{\wcS, M}(X_2; \bk'\oplus\bl')t^{\sum_{e \in P(v, v'_1)}l'_e}. 
\end{multline}
By using Proposition~\ref{contracting1} and applying Theorem~\ref{thm: explicit calc} to $(X_1,\bk')$, we have
\begin{equation}\label{eq:sh-app1}
\zeta^{}_{\wcS,M}(X_1;\bk')=\zeta^{\sh}_{\wcS,M}(\bk\sh\bl).
\end{equation}
On the other hand, by Example~\ref{ex_of_2crt1}, we see that the right hand side of  \eqref{eq:sh-app0} coincides with
\begin{equation}\label{eq:sh-app2}
(-1)^{\wt(\bl)}\sum_{\bl'\in \bZ^{\dep(\bl)}_{\geq0}}b\binom{\bl}{\bl'}
\zeta^{\sh}_{\wcS,M}(\bk, \overline{\bl+\bl'}){t}^{\wt(\bl')}.
\end{equation}
By combining \eqref{eq:sh-app0} and \eqref{eq:sh-app1} with \eqref{eq:sh-app2}, we have the desired formula.
\end{proof}
\begin{corollary}
For any indices $\bk$ and $\bl$, we have
\[
\zeta^{\sh}_{\wcS}(\bk \sh \bl)=(-1)^{\wt(\bl)}\sum_{\bl'\in \bZ^{\dep(\bl)}_{\geq0}}b\binom{\bl}{\bl'}\zeta^{\sh}_{\wcS}(\bk, \overline{\bl+\bl'}){t}^{\wt(\bl')}.
\]
\end{corollary}
\begin{proof}
Take the limit $M\to\infty$ in \eqref{sh for M} and use Theorem~\ref{lim}. 
\end{proof}
\subsection{Representation algorithm}\label{subsec3.4}
In this subsection, we prove the following: 
\begin{theorem}\label{thm:FT}
Let $X=(V, E, \rt, V_\bullet)$ be a $2$-colored rooted tree such that $\rt\in V_\bullet$  
and $\bk=(k_e)_{e\in E}$ an essentially positive index on $X$.
Then, there exists $w\in\frH^1$ such that
\[
\zeta^{}_{\wcS, M}(X; \bk)=Z^{\sh}_{\wcS,M}(w)
\]
holds.
\end{theorem}

Here the \emph{essential positivity} of an index is defined as follows. 

\begin{definition}
Let $X=(V, E, \rt, V_\bullet)$ be a $2$-colored rooted tree. An index $\bk$ on $X$ is said 
\emph{essentially positive} when the sum $\sum_{e \in P(v_1, v_2)}k_e$ is positive 
for any two distinct vertices  $v_1, v_2\in V_\bullet$.
\end{definition}

We will give an algorithm to construct the element $w\in\frH^1$ in Theorem \ref{thm:FT}. 
First let us suppose a stronger condition. 

\begin{definition}
Let $X=(V, E, \rt, V_\bullet)$ be a $2$-colored rooted tree and $\bk=(k_e)_{e\in E}$ an index on $X$.
The pair $(X, \bk)$ is called \emph{harvestable} if the following conditions hold:
\begin{description}
\item[(H1)] The root $\rt$ is a terminal of $(V,E)$. In particular, $\rt$ is in $V_\bullet$.
\item[(H2)] All elements of $V_\circ$ are branched points.
\item[(H3)] All elements of $V_\bullet$ are not branched points.
\item[(H4)] If $v\in V_\circ$ is the parent of $w\in V$, then $k_{\{v,w\}}$ is positive. 
\item[(H5)] If $v,w\in V_\bullet$ and $\{v,w\}\in E$, then $k_{\{v,w\}}$ is positive. 
\end{description}
Here, a \emph{branched point} is a vertex of degree at least $3$ and 
the \emph{parent} of a vertex $v\neq \rt$ is the unique vertex $p$ satisfying $\{v,p\}\in P(\rt,v)$. 
\end{definition}

\begin{remark}
The notions of harvestablity and essential positivity are introduced by the first author \cite{O}, 
but we add the condition (H5) here. 
Note that, for a pair $(X,\bk)$ satisfying the conditions from (H1) to (H4), 
it satisfies (H5) if and only if $\bk$ is essentially positive. 
In particular, a harvestable pair satisfies the condition of Theorem \ref{thm:FT}. 
\end{remark}

\begin{definition}\label{def: w(X,k)}
For a harvestable pair $(X, \bk)$, we define an element $w(X,\bk)$ in $\frH^1$ recursively as follows. 
\begin{enumerate}
\item For $(X,\bk)$ of the following form, we define $w(X,\bk)\coloneqq z_{k_1}\cdots z_{k_r}$. 
\[
\begin{tikzpicture}
\coordinate (v_1) at (0,2.0) node at (v_1) {};
\coordinate (v_2) at (0,1.4) node at (v_2) {};
\coordinate (v_r) at (0,0.6) node at (v_r) {};
\coordinate (v_{r+1}) at (0,0) node at (v_{r+1}) {};
\draw              (0,1.4) -- node [right] {\tiny $k_1$} (0,2.0);
\draw [dotted] (0,0.6) -- (0,1.4);
\draw              (0,0) -- node [right] {\tiny $k_r$} (0,0.6);
\fill (v_1) circle (2pt) (v_2) circle (2pt) (v_r) circle (2pt);
\fill (-0.1,-0.1) rectangle (0.1,0.1);
\end{tikzpicture}
\]
\item Let $(X,\bk)$ (resp.~$(X_j,\bk_j)$) be given by the left hand side (resp.~the right hand side) of the following diagrams. 
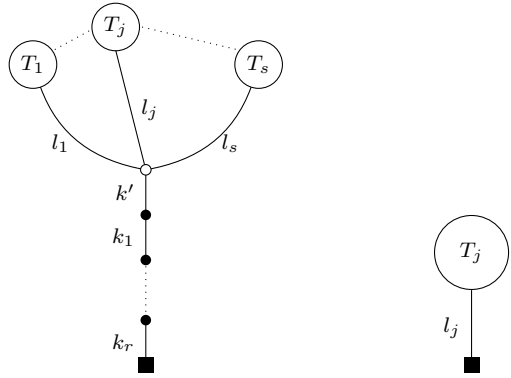
\[
\begin{tikzpicture}
\coordinate (T_1) at (-1.5,4) node at (T_1) {\tiny $T_1$};
\coordinate (T_j) at (-0.4,4.5) node at (T_j) {\tiny $T_j$};
\coordinate (T_s) at (1.5,4) node at (T_s) {\tiny $T_s$};
\coordinate (v_0) at (0,2.6) node at (0,2.6) {};
\coordinate (v_1) at (0,2.0) node at (0,2.0) {};
\coordinate (v_2) at (0,1.4)  node at (0,1.4) {};
\coordinate (v_r) at (0,0.6) node at (0,0.6) {};
\coordinate (v_{r+1}) at (0,0.0) node at (0,0) {};
\draw              (-1.4,3.7) to [out=290,in=170] node [left] {\tiny $l_1$} (v_0);
\draw [dotted] (-1.21,4.2) -- (-0.75,4.45);
\draw              (-0.4,4.19) -- node [right] {\tiny $l_j$} (v_0);
\draw [dotted] (-0.05,4.5) -- (1.21,4.2);
\draw              (1.4,3.7) to [out=250, in=10] node [right] {\tiny $l_s$} (v_0);
\draw              (v_0) -- node [left] {\tiny $k'$} (v_1);
\draw              (v_1) -- node [left] {\tiny $k_1$} (v_2);
\draw [dotted] (v_2) -- (v_r);
\draw              (v_r) -- node [left] {\tiny $k_r$} (v_{r+1});
\draw (T_1) circle [radius=9pt];
\draw (T_j) circle [radius=9pt];
\draw (T_s) circle [radius=9pt];
\fill (v_1) circle (2pt) (v_2) circle (2pt) (v_r) circle (2pt);
\fill (-0.1,-0.1) rectangle (0.1,0.1);
\filldraw[fill=white] (v_0) circle [radius=0.7mm];
\end{tikzpicture}
\hspace{2cm}
\begin{tikzpicture}
\coordinate (T_j) at (0,1.5) node at (T_j) {\tiny $T_j$};
\coordinate (rt) at (0,0.0) node at (0,0) {};
\draw (0,1) -- node [left] {\tiny $l_j$} (rt);
\draw (T_j) circle [radius=14pt];
\fill (-0.1,-0.1) rectangle (0.1,0.1);
\end{tikzpicture}
\]
Assume that $w_j=w(X_j,\bk_j)$ for $j=1,\ldots,s$ are already defined. Then we define 
\[w(X,\bk)\coloneqq (w_1\sh\cdots\sh w_s)x^{k'}z_{k_1}\cdots z_{k_r}. \]
\end{enumerate}
In fact, we see that this procedure exhausts all harvestable pairs $(X,\bk)$ 
by induction on the cardinal of $V_\circ$. 
\end{definition}

The next theorem generalizes Theorem~\ref{thm: explicit calc} to the harvestable case. 
\begin{theorem}\label{main theorem}
For any harvestable pair $(X,\bk)$, we have 
\begin{equation}\label{eq:ind4}
\zeta^{}_{\wcS, M}(X; \bk)=Z^{\sh}_{\wcS, M}\bigl(w(X,\bk)\bigr).
\end{equation}
\end{theorem}
\begin{proof}
The equation is given in Example \ref{ex_of_2crt1} if $X$ has no branched point, 
and in Theorem \ref{thm: explicit calc} if there is just one branched point. 
In the general case, one obtains the equation by making a computation similar to 
the proof of Theorem \ref{thm: explicit calc} for each branched point. 
\end{proof}

The following proposition says that any $2$-colored rooted tree with the root in $V_{\bullet}$ and an essentially positive index 
can be transformed to a harvestable pair without changing the value of the associated truncated $\wcS$-MZV.
\begin{proposition}\label{hv_form}
Let $X=(V, E, \rt, V_\bullet)$ be a $2$-colored rooted tree such that $\rt\in V_\bullet$ 
and $\bk$ an essentially positive index on $X$. 
Then there exists a harvestable pair $(X_\h,\bk_\h)$ such that 
\begin{equation}\label{hv}
\zeta^{}_{\wcS, M}(X; \bk)=\zeta^{}_{\wcS, M}(X_\h; \bk_\h). 
\end{equation}
Explicitly, we obtain such $(X_\h,\bk_\h)$ by the following procedures: 
\begin{enumerate}
\item In $(X,\bk)$, 
if there exists an edge $e$ satisfying the condition in Proposition~\ref{contracting1}, 
then contract $e$ according to the proposition. 
Repeat this until we obtain a pair $(X_1,\bk_1)$ without such $e$. 
\item In $(X_1,\bk_1)$, 
if there exists a pair of edges satisfying the condition in Proposition~\ref{contracting2}, 
then joint them according to the proposition. 
Repeat this until we obtain a pair $(X_2,\bk_2)$ without such edges. 
\item In $(X_2,\bk_2)$, 
if there exists a black branched point $v\neq\rt$, then insert a new white vertex $v'$ 
together with an edge $\{v,v'\}$ at the location of $v$, 
and replace the edges $\{v,w\}$ by $\{v',w\}$ for vertices $w$ whose parent is $v$ in the original tree. 
Set the component of the index on the new edge $\{v,v'\}$ to be zero. 
Note that this is the inverse operation of the contraction according to Proposition~\ref{contracting1}. 
Repeat this until we obtain a pair $(X_3,\bk_3)$ without such $v$. 
\item In $(X_3,\bk_3)$, 
if the root is not terminal, then insert a new white vertex and an edge at location of the root, 
in the same way as (iii). 
The result is $(X_\h,\bk_\h)$ we want to construct. 
\end{enumerate}
\end{proposition}
The following diagrams illustrate these procedures.
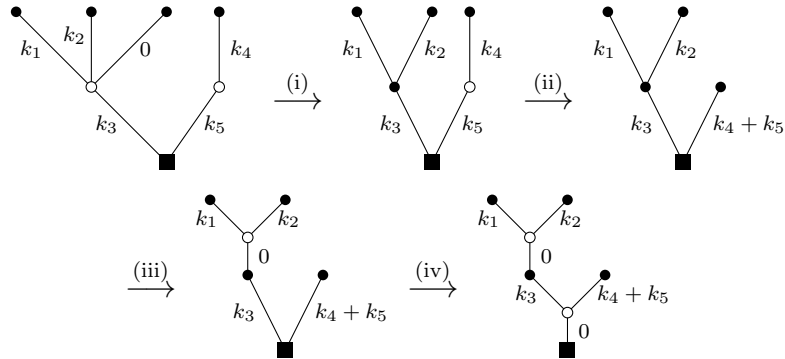
\[
\begin{tikzpicture}[baseline=20pt]
\coordinate (A) at (-2,2){};
\coordinate (B) at (-1,2){};
\coordinate (C) at (0,2){};
\coordinate (D) at (0.7,2){};
\coordinate (E) at (-1,1){};
\coordinate (F) at (0.7,1){};
\coordinate (G) at (0,0){};
\draw (A) to node [left] {\tiny $k_1$} (E); 
\draw (B) to node [above left=-2pt] {\tiny $k_2$} (E); 
\draw (C) to node [right] {\tiny $0$} (E); 
\draw (E) to node [left] {\tiny $k_3$} (G); 
\draw (D) to node [right] {\tiny $k_4$} (F); 
\draw (F) to node [right] {\tiny $k_5$} (G); 
\fill (A) circle (2pt) (B) circle (2pt) (C) circle (2pt) (D) circle (2pt);
\fill (-0.1,-0.1) rectangle (0.1,0.1);
\filldraw [fill=white] (E) circle (0.7mm) (F) circle (0.7mm);
\end{tikzpicture}
\overset{\text{(i)}}{\longrightarrow}
\begin{tikzpicture}[baseline=20pt]
\coordinate (A) at (-1,2){};
\coordinate (B) at (0,2){};
\coordinate (D) at (0.5,2){};
\coordinate (E) at (-0.5,1){};
\coordinate (F) at (0.5,1){};
\coordinate (G) at (0,0){};
\draw (A) to node [left] {\tiny $k_1$} (E); 
\draw (B) to node [right] {\tiny $k_2$} (E); 
\draw (E) to node [left] {\tiny $k_3$} (G); 
\draw (D) to node [right] {\tiny $k_4$} (F); 
\draw (F) to node [right] {\tiny $k_5$} (G); 
\fill (A) circle (2pt) (B) circle (2pt) (D) circle (2pt) (E) circle (2pt);
\fill (-0.1,-0.1) rectangle (0.1,0.1);
\filldraw [fill=white] (F) circle (0.7mm);
\end{tikzpicture}
\overset{\text{(ii)}}{\longrightarrow}
\begin{tikzpicture}[baseline=20pt]
\coordinate (A) at (-1,2){};
\coordinate (B) at (0,2){};
\coordinate (E) at (-0.5,1){};
\coordinate (F) at (0.5,1){};
\coordinate (G) at (0,0){};
\draw (A) to node [left] {\tiny $k_1$} (E); 
\draw (B) to node [right] {\tiny $k_2$} (E); 
\draw (E) to node [left] {\tiny $k_3$} (G); 
\draw (F) to node [right] {\tiny $k_4+k_5$} (G); 
\fill (A) circle (2pt) (B) circle (2pt) (E) circle (2pt) (F) circle (2pt);
\fill (-0.1,-0.1) rectangle (0.1,0.1);
\end{tikzpicture}
\]
\[
\overset{\text{(iii)}}{\longrightarrow}
\begin{tikzpicture}[baseline=20pt]
\coordinate (A) at (-1,2){};
\coordinate (B) at (0,2){};
\coordinate (C) at (-0.5,1.5){};
\coordinate (E) at (-0.5,1){};
\coordinate (F) at (0.5,1){};
\coordinate (G) at (0,0){};
\draw (A) to node [left] {\tiny $k_1$} (C); 
\draw (B) to node [right] {\tiny $k_2$} (C);
\draw (C) to node [right] {\tiny $0$} (E); 
\draw (E) to node [left] {\tiny $k_3$} (G); 
\draw (F) to node [right] {\tiny $k_4+k_5$} (G); 
\fill (A) circle (2pt) (B) circle (2pt) (E) circle (2pt) (F) circle (2pt);
\fill (-0.1,-0.1) rectangle (0.1,0.1);
\filldraw [fill=white] (C) circle (0.7mm);
\end{tikzpicture}
\overset{\text{(iv)}}{\longrightarrow}
\begin{tikzpicture}[baseline=20pt]
\coordinate (A) at (-1,2){};
\coordinate (B) at (0,2){};
\coordinate (C) at (-0.5,1.5){};
\coordinate (E) at (-0.5,1){};
\coordinate (F) at (0.5,1){};
\coordinate (D) at (0,0.5){};
\coordinate (G) at (0,0){};
\draw (A) to node [left] {\tiny $k_1$} (C); 
\draw (B) to node [right] {\tiny $k_2$} (C);
\draw (C) to node [right] {\tiny $0$} (E); 
\draw (E) to node [left] {\tiny $k_3$} (D); 
\draw (F) to node [right] {\tiny $k_4+k_5$} (D);
\draw (D) to node [right] {\tiny $0$} (G); 
\fill (A) circle (2pt) (B) circle (2pt) (E) circle (2pt) (F) circle (2pt);
\fill (-0.1,-0.1) rectangle (0.1,0.1);
\filldraw [fill=white] (C) circle (0.7mm);
\filldraw [fill=white] (D) circle (0.7mm);
\end{tikzpicture}
\]

\begin{proof}
It is sufficient to check that the pair $(X_{\h},\bk_{\h})$ obtained by the above procedures is harvestable 
since it is obvious that the equality~\eqref{hv} holds by two propositions. 

After performing (i) (resp.~(ii), resp.~(iii) and (iv)), 
the condition (H4) (resp.~(H2), resp.~(H3) and (H1)) is satisfied. 
In particular, note that the conditions (H2) and (H4) are not violated after performing (iii) and (iv) 
since the new white color vertex is branched and $\bk$ is essentially positive on $X$. 
Moreover, each procedure keeps the essential positivity of the index, 
and hence the result $(X_\h,\bk_\h)$ satisfies (H5). 
Therefore $(X_{\h},\bk_{\h})$ is harvestable. 
\end{proof}

Even if the condition $\rt\in V_\bullet$ in Theorem \ref{thm:FT} does not hold, 
we can still describe $\zeta^{}_{\wcS,M}(X;\bk)$ for any essentially positive index $\bk$ on $X$ 
in terms of values of the map $Z^{\sh}_{\wcS,M}$ as 
\[\zeta^{}_{\wcS,M}(X;\bk)=\sum_{n=0}^\infty Z^{\sh}_{\wcS,M}(w_n)\,t^n, 
\qquad w_n\in \frH^1. \]
To obtain such an expression, we first use Proposition \ref{change of root} to change the root 
to some terminal vertex, and then apply Theorem \ref{thm:FT} to each $(X_2,\bk\oplus\bl)$ 
appearing in the right hand side of \eqref{eq:change of root}. 
In particular, the limit in the following definition exists by virtue of Theorem~\ref{lim}. 

\begin{definition}\label{def:M-->oo}
Let $X=(V, E, \rt, V_\bullet)$ be a $2$-colored rooted tree and 
$\bk=(k_e)_{e\in E}$ be an essentially positive index on $X$. 
We define the \emph{$\wcS$-MZV} $\zeta^{}_{\wcS}(X;\bk)$ \emph{associated with $X$} as 
\[
\zeta^{}_{\wcS}(X;\bk)\coloneqq\lim_{M\to\infty}\zeta^{}_{\wcS,M}(X;\bk), 
\]
where the limit is taken coefficientwise in $\cZ\jump{t}$. 
\end{definition}

\section{On a refinement of the conjecture of Kaneko and Zagier}\label{sec:KZ-conj}
\subsection{The Kaneko--Zagier conjecture and its refinement}\label{subsec:KZ-conj}
The Kaneko--Zagier conjecture states that $\cA$-MZVs and $\cS$-MZVs satisfy exactly the same algebraic relations. 
More precisely: 
\begin{conjecture}[\cite{KZ}]\label{Kanenko-Zagier conjecture}
Let $\cZ_{\cA}$ be the $\bQ$-subalgebra of $\cA$ generated by all $\cA$-MZVs.
Then there is a $\bQ$-algebra isomorphism from $\cZ_\cA$ onto $\overline{\cZ}$ 
which sends $\zeta^{}_\cA(\bk)$ to $\zeta^{}_\cS(\bk)$. 
\end{conjecture}
If the above conjecture is true, then $\overline{\cZ}$ is generated by $\cS$-MZVs. 
In fact, Yasuda proved the following result without assuming the conjecture. 
\begin{theorem}[{\cite[Theorem 6.1]{Yas}}]\label{Yasuda's theorem}
For $\bullet\in\{*,\sh\}$, let $\cZ_{\cS}^{\bullet}$ be the $\bQ$-subalgebra of $\bR$ 
generated by all $\zeta_{\cS}^{\bullet}(\bk)$. Then we have 
\[
\cZ_{\cS}^{\bullet}=\cZ.
\]
\end{theorem}

As the $\wcA$-version of $\cZ_\cA\subset \cA$, we define 
\[
\cZ_{\wcA}\coloneqq 
\Biggl\{\sum_{i=1}^{\infty}a_i\zeta^{}_{\wcA}(\bk_i)\pp^{n_i}\in\wcA \ \Biggm| 
\begin{array}{l} a_i\in\bQ, \quad \bk_i\colon\text{index}, \\
n_i\in\bZ_{\geq 0} \text{ with } n_i\to\infty \  (i\to\infty)\end{array}\Biggr\}. 
\]
In the following, we use the rule that the symbol $Z$ with some suffixes denotes 
the $\bQ$-linear map determined by the corresponding zeta values with the same suffixes. 
For example, $Z_{\wcA}$ denotes the $\bQ$-linear map from $\frH^1$ to $\wcA$ 
defined by $Z_{\wcA}(z_{\bk})=\zeta_{\wcA}(\bk)$. 
Then $\cZ_{\wcA}$ is also written as 
\[\cZ_{\wcA}=\Biggl\{\sum_{n=0}^{\infty}Z_{\wcA}(w_n)\pp^n \in\wcA \ \Biggm| w_n\in\frH^1\Biggr\}. \]

We equip $\cZ_{\wcA}$ and $\overline{\cZ}\jump{t}$ with the $\pp$-adic and $t$-adic topology, respectively. 
Note that $\cZ_{\wcA}$ is the $\pp$-adically complete $\bQ$-subalgebra of $\wcA$ generated by $\wcA$-MZVs and $\pp$. 
Then Kaneko--Zagier's conjecture (Conjecture~\ref{Kanenko-Zagier conjecture}) is refined as follows:
\begin{conjecture}[{cf.~\cite[Conjecture~5.3.2]{J5}, \cite[Conjecture~2.3]{Ro2}}]\label{conj:RKZ}
There is a topological $\bQ$-algebra isomorphism from $\cZ_{\wcA}$ onto $\overline{\cZ}\jump{t}$ 
which sends $\zeta_{\wcA}^{}(\bk)$ to $\zeta_{\wcS}^{}(\bk)$ and $\pp$ to $t$.
\end{conjecture}

The $t$-adic version of Yasuda's theorem related to the above conjecture is 
due to Jarossay \cite[Proposition~5.5]{J2}.
Here we present it in our notation for the convenience of the reader.
\begin{proposition}\label{prop:GYT}
For $\bullet\in\{*,\sh\}$, we define  
\begin{align*}
\cZ_{\wcS}^\bullet&\coloneqq 
\Biggl\{\sum_{i=1}^{\infty}a_i\zeta_{\wcS}^\bullet(\bk_i)t^{n_i}\in\bR\jump{t} \ \Biggm| 
\begin{array}{l} a_i\in\bQ, \quad \bk_i\colon\text{index}, \\
n_i\in\bZ_{\geq 0} \text{ with } n_i\to\infty \  (i\to\infty)\end{array}\Biggr\}\\
&=\Biggl\{\sum_{n=0}^{\infty}Z_{\wcS}^\bullet(w_n)t^n\in\bR\jump{t} \ \Biggm| w_n\in\frH^1\Biggr\}. 
\end{align*}
Then we have
\[
\cZ_{\wcS}^{\bullet}=\cZ\jump{t}.
\]
\end{proposition}
\begin{proof}
The inclusion $\cZ_{\wcS}^{\bullet}\subset\cZ\jump{t}$ is obvious. 
To prove the opposite inclusion, let $\Xi=\sum_{n=0}^\infty \xi_n t^n$ be an arbitrary element of $\cZ\jump{t}$. 
By Theorem~\ref{Yasuda's theorem}, we have $\xi_0=Z_{\cS}^\bullet(w_0)$ for some $w_0\in\frH^1$. 
Then we have 
$\Xi\equiv \xi_0\equiv Z_{\wcS}^\bullet(w_0)\mod t$. Hence we can write 
\[\Xi-Z_{\wcS}^\bullet(w_0)=\sum_{n=1}^\infty \xi'_n t^n\in t\cZ\jump{t}. \]
Again by Theorem \ref{Yasuda's theorem}, there is some $w_1\in\frH^1$ with $\xi'_1=Z_{\cS}^\bullet(w_1)$ and 
\[\Xi-Z_{\wcS}^\bullet(w_0)-Z_{\wcS}^\bullet(w_1)t=\sum_{n=2}^\infty\xi''_n t^n\in t^2\cZ\jump{t}, \]
and so on. Repeating this procedure, we obtain a sequence $(w_n)$ in $\frH^1$ such that 
\[\Xi-Z_{\wcS}^\bullet(w_0)-Z_{\wcS}^\bullet(w_1)t-\cdots-Z_{\wcS}^\bullet(w_n) t^n\in t^{n+1}\cZ\jump{t}. \]
This amounts to the equality $\Xi=\sum_{n=0}^\infty Z_{\wcS}^\bullet(w_n) t^n$, as required. 
\end{proof}

\subsection{$\wcS$-MZVs of Mordell--Tornheim type}\label{subsec:MT}
In this subsection, we define and study the $\cS$- and $\wcS$-MZVs of \emph{Mordell--Tornheim type}, 
which correspond to Kamano's $\cA$-MZV of Mordell--Tornheim type \cite{Kam} and its natural lift to $\wcA$ 
via Conjecture \ref{Kanenko-Zagier conjecture} and Conjecture \ref{conj:RKZ}, respectively.

First let us recall the definitions of the ($\cA$-finite) multiple zeta values of Mordell--Tornheim type. 
Let $k_1, \ldots, k_r, k_{r+1}$ be non-negative integers. Suppose that at least $r$ numbers of them are positive. 
Then the \emph{multiple zeta value of Mordell--Tornheim type} $\zeta^{\MT}(k_1, \ldots, k_r; k_{r+1})$ 
(MZV of MT type, for short) is defined by the following infinite series (cf.~\cite{M,T}): 
\[
\zeta^{\MT}(k_1, \ldots, k_r; k_{r+1})
\coloneqq \sum_{m_1, \ldots, m_r>0}\frac{1}{m^{k_1}_1 \cdots m^{k_r}_r(m_1+\cdots+m_r)^{k_{r+1}}}.
\]
As the truncation of this series, we also consider 
\[
\zeta^{\MT}_M(k_1, \ldots, k_r; k_{r+1})
\coloneqq \sum_{\substack{m_1, \ldots, m_r>0\\ m_1+\cdots+m_r<M}}
\frac{1}{m^{k_1}_1 \cdots m^{k_r}_r(m_1+\cdots+m_r)^{k_{r+1}}}. 
\]
By using this truncation, Kamano \cite{Kam} introduced 
the \emph{$\cA$-finite multiple zeta value of Mordell--Tornheim type} ($\cA$-MZV of MT type) as 
\[
\zeta^\MT_{\cA}(k_1,\ldots,k_r;k_{r+1})\coloneqq 
\bigl(\zeta^\MT_p(k_1,\ldots,k_r;k_{r+1})\bmod p\bigr)_p\in\cA, 
\]
and proved that 
\[
\zeta^\MT_{\cA}(k_1,\ldots,k_r;k_{r+1})=
\begin{cases}
Z_{\cA}\bigl((z_{k_1}\sh\cdots \sh z_{k_r})x^{k_{r+1}}\bigr) & (k_1,\ldots,k_r>0),\\
Z_{\cA}\bigl((z_{k_1}\sh\cdots\sh\check{z}_{k_i}\sh\cdots\sh z_{k_r})z_{k_{r+1}}\bigr) 
& (k_i=0), 
\end{cases} 
\]
where the symbol $\check{z}_{k_i}$ means that the factor $z_{k_i}$ is skipped 
(see \cite[Theorem 2.1]{Kam}). 

In fact, if we define the \emph{$\wcA$-finite multiple zeta value of Mordell--Tornheim type} ($\wcA$-MZV of MT type) by 
\[
\zeta^\MT_{\wcA}(k_1,\ldots,k_r;k_{r+1})\coloneqq 
\bigl(\bigl(\zeta^\MT_p(k_1,\ldots,k_r;k_{r+1})\bmod p^n\bigr)_p\bigr)_n\in\wcA, 
\]
then Kamano's proof also gives 
\[
\zeta^\MT_{\wcA}(k_1,\ldots,k_r;k_{r+1})=
\begin{cases}
Z_{\wcA}\bigl((z_{k_1}\sh\cdots \sh z_{k_r})x^{k_{r+1}}\bigr) & (k_1,\ldots,k_r>0),\\
Z_{\wcA}\bigl((z_{k_1}\sh\cdots\sh\check{z}_{k_i}\sh\cdots\sh z_{k_r})z_{k_{r+1}}\bigr) 
& (k_i=0). 
\end{cases} 
\]
Therefore, in view of Conjecture \ref{Kanenko-Zagier conjecture} and Conjecture \ref{conj:RKZ}, 
the $\cS$- and $\wcS$-counterparts of these values should be defined as follows: 

\begin{definition}
Let $k_1,\ldots,k_{r+1}$ be non-negative integers which are positive with at most one exception. 
Then we set 
\[
\zeta^{\MT,\sh}_{\wcS}(k_1,\ldots,k_r;k_{r+1})
\coloneqq \begin{cases}
Z^{\sh}_{\wcS}\bigl((z_{k_1}\sh\cdots \sh z_{k_r})x^{k_{r+1}}\bigr) & (k_1,\ldots,k_r>0),\\
Z^{\sh}_{\wcS}\bigl((z_{k_1}\sh\cdots\sh\check{z}_{k_i}\sh\cdots\sh z_{k_r})z_{k_{r+1}}\bigr) & (k_i=0).  
\end{cases} 
\]
Its reduction modulo $\pi^2$ is called the \emph{$t$-adic symmetric multiple zeta value of Mordell--Tornheim type} 
($\wcS$-MZV of MT type) and denoted by $\zeta^{\MT}_{\wcS}(k_1,\ldots,k_r;k_{r+1})$. 

The truncated values $\zeta^{\MT,\sh}_{\wcS,M}(k_1,\ldots,k_r;k_{r+1})\in\bQ\jump{t}$ are similarly defined, 
and we have $\lim\limits_{M\to\infty}\zeta^{\MT,\sh}_{\wcS,M}(k_1,\ldots,k_r;k_{r+1})
=\zeta^{\MT,\sh}_{\wcS}(k_1,\ldots,k_r;k_{r+1})$ by Theorem \ref{lim}. 

The $\cS$-versions of these values, $\zeta^{\MT,\sh}_{\cS}(k_1,\ldots,k_r;k_{r+1})$ etc., 
are defined in the same way. They are also obtained by substituting $t=0$ in the corresponding $\wcS$-values. 
In particular, $\zeta^\MT_{\cS}(k_1,\ldots,k_r;k_{r+1})\in\overline{\cZ}$ is called 
the \emph{symmetric multiple zeta value of Mordell--Tornheim type} ($\cS$-MZV of MT type). 
\end{definition}

\begin{remark}
Here we do not define $\zeta^{\MT,*}_{\wcS}(k_1,\ldots,k_r;k_{r+1})$; 
the authors are not sure what definition is good, if any. 
An obvious candidate is to replace $Z^{\sh}_{\wcS}$ by $Z^{*}_{\wcS}$ in the above definition, 
but we have no non-trivial result on it except its congruence to $\zeta^{\MT,\sh}_{\wcS}(k_1,\ldots,k_r;k_{r+1})$ 
modulo $\pi^2$, which follows from Proposition \ref{prop:indep-bullet}. 
\end{remark}

In what follows, we study the $\sh$-truncated $\wcS$-MZV of MT type 
as an application of our theory developed in \S\ref{sec3}. 
Let $X$ be the $2$-colored rooted tree in the figure below: 
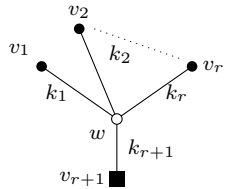
\[
\begin{tikzpicture}
\coordinate (v_1) at (-1,1.5) node at (v_1) [above left] {\tiny $v_1$};
\coordinate (v_2) at (-0.5,2.0) node at (v_2) [above] {\tiny $v_2$};
\coordinate (v_r) at (1,1.5) node at (v_r) [right] {\tiny $v_r$};
\coordinate (w) at (0,0.8) node at (w) [below left] {\tiny $w$};
\coordinate (v_{r+1}) at (0,0) node at (v_{r+1}) [left] {\tiny $v_{r+1}$};
\draw (v_1) -- node [left] {\tiny $k_1$} (w);
\draw (v_2) -- node [above right] {\tiny $k_2$} (w);
\draw (v_r) -- node [right] {\tiny $k_r$} (w);
\draw (w) -- node [right] {\tiny $k_{r+1}$} (v_{r+1});
\draw [dotted] (-0.3,2.0) -- (0.85,1.6);
\fill (v_1) circle (2pt) (v_2) circle (2pt) (v_r) circle (2pt);
\fill (-0.1,-0.1) rectangle (0.1,0.1);
\filldraw[fill=white] (w) circle[radius=0.7mm];
\end{tikzpicture}
\]
We regard an $(r+1)$-ple $\bk=(k_1,\ldots,k_{r+1})$ of non-negative integers as an index on $X$ 
as indicated in the figure. Note that $\bk$ is essentially positive if and only if 
the integers $k_1,\ldots,k_{r+1}$ are positive with at most one exception. 

\begin{proposition}\label{prop:MT tree}
Let $\bk=(k_1,\ldots,k_{r+1})$ be an essentially positive index on $X$. Then we have 
\[
\zeta^{\MT,\sh}_{\wcS, M}(k_1, \ldots, k_r; k_{r+1})=\zeta^{}_{\wcS,M}(X;\bk),
\]
and hence 
\[
\zeta^{\MT,\sh}_{\wcS}(k_1, \ldots, k_r; k_{r+1})=\zeta^{}_{\wcS}(X;\bk).
\]
\end{proposition}
\begin{proof}
If $k_1,\ldots,k_r>0$, we can apply Theorem \ref{main theorem} or Theorem \ref{thm: explicit calc} to obtain the formula. 
If $k_i=0$ for some $i\in\{1,\ldots,r\}$, we first apply the algorithm of Proposition \ref{hv_form} 
with $\rt_{\h}=\rt=v_{r+1}$, and then use Theorem \ref{main theorem} or Theorem \ref{thm: explicit calc}. 
\end{proof}
From the definition of $\zeta^{}_{\wcS,M}(X;\bk)$, we can compute the $t$-adic expansion of 
the $\sh$-truncated $\wcS$-MZV of MT type. 
The coefficients are described in terms of the truncated MZVs of MT type as follows. 
\begin{proposition}\label{prop:MT t-expansion}
Let $k_1,\ldots,k_{r+1}$ be non-negative integers which are positive with at most one exception. 
Then, we have 
\begin{align*}
\zeta^{\MT,\sh}_{\wcS,M}(k_1,\ldots,k_r;k_{r+1})
&=\zeta^\MT_M(k_1,\ldots,k_r;k_{r+1})\\
&\quad+\sum_{i=1}^r(-1)^{k_i+k_{r+1}}\sum_{l,l'=0}^\infty \binom{k_i+l-1}{l}\binom{k_{r+1}+l'-1}{l'}\\
&\hspace{3cm} \times\zeta^\MT_M(k_1,\ldots,\check{k}_i,\ldots,k_r,k_{r+1}+l';k_i+l)\; t^{l+l'}. 
\end{align*}
By taking the limit $M\to\infty$, we also have the same identity without truncation. 
\end{proposition}
\begin{proof}
Set $m_i\coloneqq m_{v_i}$ ($1 \leq i \leq r+1$), 
$M_{i,j}\coloneqq m_i+\cdots+m_j$ and $M_j\coloneqq M_{1,j}$ as in \eqref{eq:M_ij}, 
and let $e_a$ denotes the edge $\{w,v_a\}$. 
Then we have 
\[
L_{e_a}\bigl(X,v_i; (m_v)_{v\in V_{\bullet}}\bigr)=
\begin{cases}
m_a & (a\ne i),\\
m_a+t & (a=i)
\end{cases}
\]
for $1\leq a\leq r$ and 
\[
L_{e_{r+1}}\bigl(X,v_i; (m_v)_{v\in V_{\bullet}}\bigr)=
\begin{cases}
M_r & (r+1=i),\\
M_r+t & (r+1\ne i). 
\end{cases}
\]
Thus we obtain 
\[
\zeta^{}_M(X, v_{r+1}; \bk)
=\sum_{\substack{m_1, \ldots, m_r>0 \\ M_r < M}}
\frac{1}{{m\mathstrut}_1^{k_1}\cdots {m\mathstrut}_r^{k_r}{M\mathstrut}_r^{k_{r+1}}}
=\zeta^\MT_M(k_1,\ldots,k_r;k_{r+1}) 
\]
and, for $1\leq i\leq r$, 
\begin{align*}
&\zeta^{}_M(X, v_i; \bk)\\
&=\sum_{\substack{m_1, \ldots, m_{i-1}, m_{i+1}, \ldots, m_{r+1} \geq 1 \\ -M < m_i <0 \\ M_{r+1}=0}}
\frac{1}{{m\mathstrut}_1^{k_1}\cdots {m\mathstrut}_{i-1}^{k_{i-1}}{(m_i+t)\mathstrut}^{k_i}
{m\mathstrut}_{i+1}^{k_{i+1}}\cdots {m\mathstrut}_r^{k_r}{(M_r+t)\mathstrut}^{k_{r+1}}}\\
&=\sum_{\substack{m_1, \ldots, m_{i-1}, m_{i+1}, \ldots, m_{r+1} \geq 1 \\ M_{i-1}+M_{i+1, r+1}< M}}
\frac{1}{{m\mathstrut}_1^{k_1}\cdots {m\mathstrut}_{i-1}^{k_{i-1}}
{(t-M_{i-1}-M_{i+1, r+1})\mathstrut}^{k_i}
{m\mathstrut}_{i+1}^{k_{i+1}}\cdots {m\mathstrut}_r^{k_r}{(t-m_{r+1})\mathstrut}^{k_{r+1}}}. 
\end{align*}
By expanding $(t-M_{i-1}-M_{i+1,r+1})^{-k_i}$ and $(t-m_{r+1})^{-k_{r+1}}$ in $t$, we see that 
\begin{multline*}
\zeta_M(X,v_i;\bk)
=(-1)^{k_i+k_{r+1}}\sum_{l,l'=0}^\infty \binom{k_i+l-1}{l}\binom{k_{r+1}+l'-1}{l'}\\
\times\zeta^\MT_M(k_1,\ldots,\check{k}_i,\ldots,k_r,k_{r+1}+l';k_i+l)\; t^{l+l'}. 
\end{multline*}
Therefore, the summation on $i=1,\ldots,r+1$ gives the desired formula. 
\end{proof}

\begin{remark}
By setting $t=0$ in Proposition \ref{prop:MT t-expansion}, we have 
\[
\zeta^{\mathrm{MT},\sh}_{\cS}(k_1,\dots,k_r;k_{r+1})
=\sum_{i=1}^{r+1}(-1)^{k_i+k_{r+1}}\zeta^\MT(k_1,\ldots,\check{k}_i,\ldots,k_{r+1};k_i).
\]
This is equal to $(-1)^{k_{r+1}}\Omega(\bk)$ with $\bk=(k_1,\ldots,k_{r+1})$, where $\Omega(\bk)$ denotes 
the symmetric MZV of Mordell--Tornheim type recently introduced by Bachmann--Takeyama--Tasaka \cite[\S4.5]{BTT}. 
\end{remark}

Proposition \ref{change of root} gives the following relation among 
the $\sh$-truncated $\wcS$-MZVs of MT type. 
Note that this involves an infinite series and is \emph{not} a linear relation in the algebraic sense. 

\begin{proposition}\label{prop:Kam3.1}
Let $k_1,\dots,k_r,k_{r+1}$ be non-negative integers which are positive with at most one exception. 
Then we have
\begin{equation}\label{eq:root_change_wcSMT}
\begin{split}
\zeta^{\MT,\sh}_{\wcS,M}(k_1,\dots,k_r;k_{r+1})
&=(-1)^{k_1+k_{r+1}}\sum_{l,l' \ge0}\binom{k_1+l-1}{l}\binom{k_{r+1}+l'-1}{l'}\\
&\hspace{3cm} \times\zeta^{\MT,\sh}_{\wcS,M}(k_{r+1}+l',k_2,\dots,k_r;k_{1}+l)t^{l+l'}.
\end{split}
\end{equation}
\end{proposition}
\begin{proof}
By swapping the root $v_{r+1}$ with $v_1$ according to Proposition~\ref{change of root}, 
we have the desired formula.
\end{proof}

By writing down both sides of \eqref{eq:root_change_wcSMT} in terms of $Z^{\sh}_{\wcS,M}$, 
we obtain an identity of series involving infinitely many $\sh$-truncated $\wcS$-MZVs. 
In particular, by setting $t=0$, this gives a linear relation among $\sh$-truncated $\cS$-MZVs. 
This is an analogue of \cite[Theorem 3.2]{Kam}. 

\begin{corollary}
Let $k_1,\ldots,k_r$ be positive integers. 
Then we have 
\[
Z^{\sh}_{\cS, M}(z_{k_1} \sh\cdots\sh z_{k_r})
=(-1)^{k_1}Z^{\sh}_{\cS,M}((z_{k_2}\sh\cdots\sh z_{k_r})z_{k_1})
\]
and 
\[
Z^{\sh}_{\cS, M}((z_{k_1} \sh\cdots\sh z_{k_r})x^l)
=(-1)^{k_1+l}Z^{\sh}_{\cS, M}((z_{k_2} \sh \cdots \sh z_{k_r}\sh z_l)x^{k_1})
\]
for any integer $l\geq 1$. 
\end{corollary}


\end{document}